\documentclass[11pt]{article}

\usepackage{fancyhdr}
\usepackage{amsfonts}
\usepackage{amsmath}
\usepackage{amssymb}
\usepackage{amsthm}
\usepackage{tikz}
\usepackage{amscd}
\usepackage{amstext}
\usepackage{verbatim}

\usepackage[draft=false,colorlinks,bookmarksnumbered,linkcolor=black, citecolor=black]{hyperref}

\usepackage{oldgerm}
\usepackage{graphics}
\usepackage{graphicx}

\newlength{\fixboxwidth}
\newcommand*{\distr}[2]{\left\langle #1, #2 \right\rangle}              
\newcommand*{\distance}[2]{\mathrm{dist}\!\left( #1, #2 \right)}            
\newcommand*{\abs}[1]{\left| #1 \right|}                                
\newcommand*{\norm}[1]{\left\| #1 \right\|}                             
\newcommand*{\sep}{\; \vrule \;}                                        
\renewcommand{\d}{\,\mathrm{d}}											
\newcommand{\osc}{\mathrm{osc}}											
\newcommand{\loc}{\mathrm{loc}}											
\newcommand{\esssup}{ \mathop{\mathrm{ess}\text{-}\mathrm{sup}}\limits }                 

\newcommand{\re}{\mathbb{R}}\newcommand{\N}{\mathbb{N}}
\newcommand{\zz}{\mathbb{Z}}

\newcommand{\com}{\mathbb{C}}
\newcommand{\Z}{{\zz}^d}

\newcommand{\R}{{\re}^d}

\newcommand{\cs}{{\mathcal S}}

\newcommand{\cl}{{\mathcal L}}

\newcommand{\cf}{{\mathcal F}}
\newcommand{\cfi}{{\cf}^{-1}}

\newcommand{\supp}{{\rm supp \, }}

\newcommand{\be}{\begin{equation}}
\newcommand{\ee}{\end{equation}}
\newcommand{\beq}{\begin{eqnarray}}
\newcommand{\beqq}{\begin{eqnarray*}}
\newcommand{\eeq}{\end{eqnarray}}
\newcommand{\eeqq}{\end{eqnarray*}}

\newtheorem{satz}{Theorem}

\newtheorem{rem}{Remark}
\newtheorem{defi}{Definition}
\newtheorem{lem}{Lemma}
 
\newtheorem{prop}{Proposition}

\begin{document}


\title{Oscillations and Differences in\\ Triebel-Lizorkin-Morrey Spaces}
\author{Marc Hovemann\footnote{Philipps-University Marburg, Institute of Mathematics, Hans-Meerwein-Stra{\ss}e 6, 35043 Marburg, Germany. 
Email: \href{mailto:hovemann@mathematik.uni-marburg.de}{hovemann@mathematik.uni-marburg.de} } $^{,}$\thanks{Marc Hovemann has been supported by Deutsche Forschungsgemeinschaft (DFG), grant DA~360/24-1.}
\quad\ and \quad Markus Weimar\footnote{Julius-Maximilians-Universit\"at W\"urzburg (JMU), 
Institute of Mathematics, 
Emil-Fischer-Stra{\ss}e 30, 97074 W\"urzburg, Germany. 
Email: \href{mailto:markus.weimar@uni-wuerzburg.de}{markus.weimar@uni-wuerzburg.de} } $^{,}$\footnote{Corresponding author.}
}
\date{\today}

\maketitle

\noindent\textbf{Abstract:} In this paper we are concerned with Triebel-Lizorkin-Morrey spaces $\mathcal{E}^{s}_{u,p,q}(\Omega)$ of positive smoothness $s$ defined on (special or bounded) Lipschitz domains $\Omega\subset\R$ as well as on $\R$. 
For those spaces we prove new equivalent characterizations in terms of local oscillations which hold as long as some standard conditions on the parameters are fulfilled. 
As a byproduct, we also obtain novel characterizations of $\mathcal{E}^{s}_{u,p,q}(\Omega)$ using differences of higher order.
Special cases include standard Triebel-Lizorkin spaces $F^s_{p,q} (\Omega)$ and hence classical $L_p$-Sobolev spaces $H^s_p(\Omega)$.

\vspace{0,2 cm}


\vspace{0,2 cm}


\section{Introduction and Main Results}

Nowadays Triebel-Lizorkin spaces $F^s_{p,q} (\R)$ are well-established tools to describe the regularity of functions and distributions beyond the classical scale of $L_p$-Sobolev spaces $H^s_p(\R)$ which are included as special cases $q=2$. They have been introduced around 1970 by Lizorkin~\cite{Liz1, Liz2} and Triebel~\cite{Tr73}. 
Later these function spaces have been investigated in detail in the books of Triebel~\cite{Tr83,Tr92,Tr06} which contain numerous other historical references. 
In recent years a growing number of authors works with further generalisations of the Triebel-Lizorkin scale defined upon \emph{Morrey spaces} instead of Lebesgue spaces $L_p$. So, Triebel-Lizorkin-Morrey spaces $ \mathcal{E}^{s}_{u,p,q}(\R) $ with $ 0 < p \leq u < \infty$, $ 0 < q \leq \infty $, $ s \in \mathbb{R} $ and Triebel-Lizorkin-type spaces $ F^{s,\tau}_{p,q}(\R) $ with $ 0 < p < \infty $, $ 0 < q \leq \infty $, $ s \in \mathbb{R} $, $ 0\leq \tau<\infty $ attracted a lot of attention. 
The spaces $ \mathcal{E}^{s}_{u,p,q}(\R) $ have been introduced by Tang and Xu in 2005, see \cite{TangXu}, while  
$F^{s,\tau}_{p,q}(\R) $ showed up for the first time in 2008 in some papers of Yang and Yuan~\cite{yy1,yy2}. 
Later on, using a different notation, the latter also appeared in~\cite{Tr14}. 
Although the spaces $ \mathcal{E}^{s}_{u,p,q}(\R) $ and  $ F^{s,\tau}_{p,q}(\R) $ are defined quite differently, they have a lot of properties in common. 
Moreover, under certain conditions on the parameters they even coincide \cite{ysy}.

When it comes to applications in the theory of quasi-linear partial differential equations (PDEs), locally weighted $L_p$ averages of derivatives (such as norms in Morrey spaces) became a standard tool from the early beginning~\cite{Mor}, see also \cite{DahDieHar+} and the references therein.
In this context it seems advantageous to deal with Triebel-Lizorkin-Morrey spaces~$\mathcal{E}^{s}_{u,p,q}(\Omega)$ defined on domains $\Omega \subseteq \mathbb{R}^d$ with preferably mild restrictions on the regularity of their boundary. 
Thereby, in what follows (mostly) we will concentrate on special or bounded \emph{Lipschitz domains} $\Omega$. 
It is one main goal of this paper to prove equivalent intrinsic characterizations of the spaces $ \mathcal{E}^{s}_{u,p,q}(\Omega) $ in terms of oscillations $\osc_{v, \Omega}^{N}f$. 
Here for a function $f \in L_v^{\loc}(\Omega)$ with $0< v \leq \infty$ its local $v$-oscillation of order $N\in\N_0$ is given by
\begin{align*}
    \osc_{v, \Omega}^{N}f(x,t) := \inf_{P \in \mathcal{P}_{N}} \Big( t^{-d} \int_{B(x,t) \cap \Omega} \abs{f(y) - P(y)}^{v} \d y \Big)^{\frac{1}{v}},
    \qquad x\in\Omega, t>0,
\end{align*}
whereby $ \mathcal{P}_{N} $ denotes the set of polynomials with degree at most $N$ and for $v = \infty$ the usual modifications have to be made. If $\Omega=\R$, we simply write $\osc_{v}^{N}f:=\osc_{v,\R}^{N}f$. 
We investigate under which conditions on the parameters $s$, $u$, $p$, $q$, $v$, $N$, and $d$ functions $f\!\in~\!\!\!\mathcal{E}^{s}_{u,p,q}(\Omega)$ can be described by using~$\osc_{v, \Omega}^{N}f$ only. 
Such characterizations allow to describe Triebel-Lizorkin-Morrey spaces in terms of the decay of bestapproximation errors w.r.t.\ polynomials illustrating their strong relation to approximation theory which is well-known for classical Sobolev and Besov spaces; cf., e.g.,~\cite{CioWei} or \cite{Tr89}.
In the long run, this can be used to derive sharp assertions on the regularity of PDE solutions as it was done, e.g., in~\cite{BalDieWei} for the original Triebel-Lizorkin spaces $F^s_{p,q}(\Omega)$.
Moreover, oscillation characterizations pave the way for local error estimators needed in the construction of powerful adaptive $hp$-finite element methods on Triebel-Lizorkin-Morrey spaces $\mathcal{E}^{s}_{u,p,q}(\Omega) $. Here we refer to~\cite{CaNoSte} for the case of classical Sobolev spaces.

For Triebel-Lizorkin spaces $F^s_{p,q}(\Omega)$, whereat $\Omega\subseteq\R$ is either $\mathbb{R}^d$ or a bounded \mbox{($C^{\infty}$-)} domain, characterizations in terms of oscillations have a long history. 
Let us refer to Dorronsoro~\cite{Dor1}, Seeger~\cite{See1}, Shvartsman~\cite{Shv06}, and Triebel~\cite{Tr89} at least. 
Much more references (also concerning Besov spaces) can be found in \cite[Section 1.7.3]{Tr92}.
One such result reads as follows.

\begin{satz}[{Triebel~\cite[Theorems 3.5.1 and 5.2.1]{Tr92}}]\label{thm_hist1}
    For $d\in\N$ let $\Omega\subseteq\R$ be either $\mathbb{R}^d$ or a bounded $C^{\infty}$-domain. 
    Let $0 < p < \infty$, $0 < q \leq \infty $, $ 1 \leq v \leq \infty$, $N \in \mathbb{N}$, and $s \in \mathbb{R}$ with
    \begin{align*}
        d \max\!\left\{ 0, \frac{1}{p} - \frac{1}{v}, \frac{1}{q} - \frac{1}{v}  \right\} < s < N. 
    \end{align*}
    Then $f \in L_{\max\{p,v\}}(\Omega)$ belongs to $F^{s}_{p,q}(\Omega)$ if and only if 
    \begin{align*}
        \norm{ f \sep F^{s}_{p,q}(\Omega)}_{\osc}^{(1,v,N)} 
        := \norm{ f \sep L_{p}(\Omega) } + \norm{ \Big ( \int_{0}^{1} \big[t^{-s} \, \osc_{v, \Omega}^{N-1}f(x,t)\big]^q \frac{\d t}{t} \Big )^{\frac{1}{q}} \sep L_{p}(\Omega)} < \infty,
    \end{align*}
    with the usual modification if $q = \infty$. 
    Moreover, $\norm{ \,\cdot \sep F^{s}_{p,q}(\Omega)}_{\osc}^{(1,v,N)}$ provides an equivalent quasi-norm. 
\end{satz}

When we turn to the Triebel-Lizorkin-Morrey spaces~$\mathcal{E}^s_{u,p,q}(\Omega)$ some first equivalent descriptions in terms of local oscillations are known for the special case $\Omega = \mathbb{R}^d$; see, e.g., \cite[Sections 4.4.3 and 4.5.3]{ysy}. 
Some of those characterizations will be recalled in \autoref{re_sec_osc_Rd} below. 
Our first main result of this paper extends these assertions. It
provides new equivalent quasi-norms in terms of oscillations for the spaces $\mathcal{E}^s_{u,p,q}(\Omega)$, at which $\Omega$ is either $\mathbb{R}^d$ or a (special or bounded) Lipschitz domain. 

\begin{satz}[Oscillation Characterizations]\label{mainresult1}
    Let $0< p \leq u < \infty$, $0 < q,T,v \leq \infty$, $d,N \in \mathbb{N}$, and $0<R<\infty$. Moreover, assume that $s\in\re$ satisfies
    \begin{align}\label{cond_on_s}
		d\, \max\! \left\{ 0, \frac{1}{p} - 1, \frac{1}{q} - 1, \frac{1}{p} - \frac{1}{v}, \frac{1}{q} - \frac{1}{v} \right\} < s < N.
    \end{align}
    Then the following statements hold:
    \begin{enumerate}
        \item $\mathcal{E}^s_{u,p,q}(\R)$ is the collection of all $f \in L^{\loc}_{\max\{1,p,v\}}(\R)$ for which
        \begin{align*}
            \norm{ \bigg( \int_{B(\,\cdot\,, R)} \abs{f(y)}^v \d y \bigg)^{\frac{1}{v}} \sep \mathcal{M}^{u}_{p}(\R)}  + \norm{\bigg( \int_{0}^{T} \big[ t^{-s} \, \osc_{v}^{N-1} f(\cdot,t) \big]^{q} \frac{\d t}{t} \bigg)^{\frac{1}{q}} \sep \mathcal{M}^{u}_{p}(\R)}
        \end{align*}
        is finite (equivalent quasi-norm).
        Furthermore, the assertion remains valid when $\norm{ \big( \int_{B(\,\cdot\,, R)} \abs{f(y)}^v \d y \big)^{\frac{1}{v}} \sep \mathcal{M}^{u}_{p}(\R)}$ is replaced by $\norm{f \sep \mathcal{M}^{u}_{p}(\R)}$.

        \item In addition assume that $v\geq 1$ and let $\Omega$ be either a special or a bounded Lipschitz domain in $\R$. 
        Then $\mathcal{E}^s_{u,p,q}(\Omega)$ is the collection of all $f \in L^{\loc}_{\max\{p,v\}}(\Omega)$ for which
        \begin{align*}
            \norm{ \bigg( \int_{B(\,\cdot\,, R)\cap\Omega} \abs{f(y)}^v \d y \bigg)^{\frac{1}{v}} \sep \mathcal{M}^{u}_{p}(\Omega)} + \norm{\bigg( \int_{0}^{T} \big[ t^{-s} \, \osc_{v,\Omega}^{N-1} f(\cdot,t) \big]^{q} \frac{\d t}{t} \bigg)^{\frac{1}{q}} \sep \mathcal{M}^{u}_{p}(\Omega)}
        \end{align*}
        is finite (equivalent quasi-norm).
        If additionally $p\geq 1$, this statement remains true when $\norm{ \big( \int_{B(\,\cdot\,, R)\cap\Omega} \abs{f(y)}^v \d y \big)^{\frac{1}{v}} \sep \mathcal{M}^{u}_{p}(\Omega)}$ is replaced by $\norm{f \sep \mathcal{M}^{u}_{p}(\Omega)}$.
      \end{enumerate}
      In both cases, the usual modifications have to be made if $q=\infty$ and/or $v=\infty$.
\end{satz}
When we concentrate on the case $\Omega = \mathbb{R}^d$, \autoref{mainresult1} can be seen as a continuation of \cite[Chapter 4.4.3]{ysy} and \cite[Section~8.2]{LiYYSaU}, to cover a larger range of the parameters and to supply a clearly arranged quasi-norm.
If, in contrast, $\Omega$ is a special or bounded Lipschitz domain, the special case $p=u$ particularly provides new characterizations for the original Triebel-Lizorkin spaces $F^{s}_{p,q}(\Omega)$ (and hence for ordinary Sobolev spaces). 
Let us stress that for Lipschitz domains and $p \not = u$, to the best of our knowledge, there are no counterparts of \autoref{mainresult1} in the literature up to now.
 
Studying the restrictions on $s$ given in \eqref{cond_on_s} yields that the Morrey parameter $u$ is not showing up at all. A similar observation already has been made in \cite{Ho1} when proving characterizations of $\mathcal{E}^s_{u,p,q}(\mathbb{R}^d)$ in terms of higher order differences $\Delta_{h}^{N}f$; see \eqref{eq_diff_def} below for a precise definition.
In that paper also further investigations concerning the necessity of those conditions in the context of differences can be found. However, it is known since decades that local oscillations and differences are closely related to each other. 
Indeed, higher order differences have been an important tool for some of our proofs as well. 
Therefore, as a byproduct, we also obtain new characterizations in terms of higher order differences for Triebel-Lizorkin-Morrey spaces $\mathcal{E}^s_{u,p,q}$. 
The corresponding result reads as follows. 
\begin{satz}[Difference Characterizations]\label{thm_main_2}
    Let $0< p \leq u < \infty$, $0 < q,T,v \leq \infty$, $d,N \in \mathbb{N}$, and $0<R<\infty$. Moreover, assume that $s\in\re$ satisfies
    \begin{align}\label{cond_on_s2}
		d\, \max\! \left\{ 0, \frac{1}{p}-1, \frac{1}{q}-1, \frac{1}{p} - \frac{1}{v}, \frac{1}{q} - \frac{1}{v} \right\} < s < N.
    \end{align}
    Then the following assertions hold true:
    \begin{enumerate}
        \item If additionally $T \geq 1$, then $\mathcal{E}^s_{u,p,q}(\R)$ is the set of all $f \in L^{\loc}_{\max\{1,p,v\}}(\mathbb{R}^d)$ for which
        \begin{align*} 
            & \norm{\bigg(  \int_{B(\,\cdot\,, R)}     \abs{f(y)}^{v} \d y  \bigg)^{\frac{1}{v}}  \sep \mathcal{M}^{u}_{p}(\R)} \\
            & \qquad \qquad + \norm{\bigg(  \int_{0}^{T}  t^{-sq} \Big( t^{-d} \int_{B(0,t)} \abs{\Delta^{N}_{h}f(\cdot)}^{v} \d h \Big)^{\frac{q}{v}} \frac{\d t}{t} \bigg)^{\frac{1}{q}} \sep \mathcal{M}^{u}_{p}(\R)}
        \end{align*}
        is finite (equivalent quasi-norm). Furthermore, the assertion remains valid when $\norm{ \big( \int_{B(\,\cdot\,, R)} \abs{f(y)}^v \d y \big)^{\frac{1}{v}} \sep \mathcal{M}^{u}_{p}(\R)}$ is replaced by $\norm{f \sep \mathcal{M}^{u}_{p}(\R)}$.

        \item In addition, assume that $v\geq 1$ and let $\Omega$ be a special Lipschitz domain in $\R$. 
        Then $\mathcal{E}^s_{u,p,q}(\Omega)$ is the collection of all $f \in L^{\loc}_{\max\{p,v\}}(\Omega)$ for which
        \begin{align*} 
            & \norm{\bigg(  \int_{B(\cdot, R) \cap \Omega}     \abs{f(y)}^{v} \d y  \bigg)^{\frac{1}{v}}  \sep \mathcal{M}^{u}_{p}(\Omega)} \\
            & \qquad \qquad +  \norm{\bigg(  \int_{0}^{T}  t^{-sq} \Big( t^{-d} \int_{V^{N}(\,\cdot\,,t)} \abs{\Delta^{N}_{h}f(\cdot)}^{v} \d h \Big)^{\frac{q}{v}} \frac{\d t}{t} \bigg)^{\frac{1}{q}} \sep \mathcal{M}^{u}_{p}(\Omega)}
        \end{align*}
        is finite (equivalent quasi-norm).
        If additionally $p\geq 1$, this statement remains true when $\norm{ \big( \int_{B(\,\cdot\,, R)\cap\Omega} \abs{f(y)}^v \d y \big)^{\frac{1}{v}} \sep \mathcal{M}^{u}_{p}(\Omega)}$ is replaced by $\norm{f \sep \mathcal{M}^{u}_{p}(\Omega)}$.

        \item Additionally, assume that $\Omega$ is a bounded convex Lipschitz domain in $\R$ and let $p,q>1$ as well as $v=\infty$ such that \eqref{cond_on_s2} reduces to
        \begin{align*}
   		d\, \max\! \left\{ \frac{1}{p}, \frac{1}{q} \right\} < s < N.
        \end{align*}    
        Then $\mathcal{E}^s_{u,p,q}(\Omega)$ is the collection of all $f \in L^{\loc}_{\infty}(\Omega)$ for which
        \begin{align*} 
            & \norm{ \esssup_{y\in B(\,\cdot\,, R) \cap \Omega} \abs{f(y)} \sep \mathcal{M}^{u}_{p}( \Omega)} 
            +  \norm{\bigg(  \int_{0}^{T}  t^{-sq} \Big( \esssup_{h\in V^N(\cdot,t)}\! \abs{\Delta^{N}_{h} f(\cdot)} \Big)^{q} \frac{\d t}{t} \bigg)^{\frac{1}{q}} \sep \mathcal{M}^{u}_{p}(\Omega)}  
        \end{align*}
        is finite (equivalent quasi-norm).
        Moreover, the assertion remains valid when $\big\Vert \esssup_{y\in B(\,\cdot\,, R) \cap \Omega} \abs{f(y)} \big| \mathcal{M}^{u}_{p}( \Omega) \big\Vert$ is replaced by $\norm{f \sep \mathcal{M}^{u}_{p}(\Omega)}$.
        \end{enumerate}
    Therein we set $V^{N}(x,t) := \{ h \in \mathbb{R}^d \sep  \abs{h} < t \ \mbox{and} \ x + \ell h \in \Omega \ \mbox{for all} \ 0 \leq \ell \leq N\}$ for $t>0$.
    In all cases, the usual modifications have to be made if $q=\infty$ and/or $v=\infty$.
\end{satz}

Similar results for the special case $\Omega = \mathbb{R}^d$ can already be found in the literature; see, e.g., \cite{Ho1} and the references listed in \autoref{sect:diff_Rd}. 
Note that the conditions on the smoothness parameter $s$ in~\eqref{cond_on_s2} of \autoref{thm_main_2} are exactly the same as in~\eqref{cond_on_s} from \autoref{mainresult1}. 
For a discussion concerning necessity of those conditions we refer to \cite{Ho1}, see also \cite{HoN,HoSi20}.

This paper is organized as follows. 
In \autoref{sect:prelim} we recall the definition of Morrey and Triebel-Lizorkin-Morrey spaces defined on both $\mathbb{R}^d$ and domains. 
Moreover, here we collect some useful properties of those spaces.
\autoref{sect:quasi-norms} contains the definition of our new quasi-(semi)norms which are frequently used later on.
The Sections \ref{sect:characterizations_Rd} and \ref{sect:characterizations_domains} are devoted to the 
intrinsic characterizations of~$\mathcal{E}^{s}_{u,p,q}$ on $\R$ and Lipschitz domains $\Omega$, respectively, in terms of local oscillations and higher order differences. 
Here we prove parts (i) and (ii) of our main Theorems \ref{mainresult1} and \ref{thm_main_2}. 
In this context we also recall some results concerning this topic that are already known. 
Finally, in \autoref{Sec_Diff1_re} we derive the equivalent descriptions via higher order differences of the spaces $\mathcal{E}^{s}_{u,p,q}(\Omega)$ defined on bounded convex Lipschitz domains $\Omega$ which are stated in \autoref{thm_main_2}(iii) above.  
However, first of all we shall fix some notation.

\medskip

\noindent\textbf{Notation:} As usual, $\N$ denotes the natural numbers, $\N_0:=\N\cup\{0\}$, $\zz$ describes the integers and~$\re$ the real numbers. Further, $\R$ with $d\in\N$ denotes the $d$-dimensional Euclidean space and we put
$$
    B(x,t) := \left\{y\in \R \sep \abs{x-y}< t \right\}\, , \qquad x \in \R,\,\; t>0.
$$
All functions are assumed to be complex-valued, i.e.\ we consider functions $f\colon \R \to \com$. 
We let $\mathcal{S}(\R)$ be the collection of all Schwartz functions on $\R$ endowed with the usual topology and by $\mathcal{S}'(\R)$ we denote its topological dual, the space of all bounded linear functionals on~$\mathcal{S}(\R)$ equipped with the weak-$\ast$ topology. 
The symbol $\cf$ refers to the Fourier transform and $\cfi$ to its inverse, both defined on $\cs'(\R)$. 
For domains (open connected sets) $\Omega\subseteq\R$ and $0<v\leq \infty$, by $L_v^\loc(\Omega)$ we mean the set of locally $v$-integrable (or locally essentially bounded) functions on $\Omega$. Furthermore, $\mathcal{D}(\Omega)=C_0^\infty(\Omega)$ denotes the set of infinitely often differentiable functions with compact support on $\Omega$. 
Its topological dual, $\mathcal{D}'(\Omega)$, is the space of distributions on $\Omega$.
Almost all function spaces considered in this paper are subspaces of regular distributions from $\cs'(\R)$ or $\mathcal{D}'(\Omega)$, interpreted as spaces of equivalence classes with respect to almost everywhere equality. 
Given two quasi-Banach spaces $X$ and $Y$, the norm of a linear operator $T\colon X\to Y$ is denoted by $\norm{T \sep \cl (X,Y)}$. 
Moreover, we write $X \hookrightarrow Y$ if the natural embedding of $X$ into $Y$ is continuous. 
For $0<p<\infty$ and $0<q\leq \infty$ we shall use the well-established quantities
$$
   \sigma_p:= d\,  \max\!\left\{0, \frac 1p - 1\right\} \qquad \text{and}\qquad 
    \sigma_{p,q}:= d\, \max\!\left\{0, \frac 1p -1 , \frac 1q - 1 \right\}.
$$
The symbols $C, C_1, c, c_{1}, \ldots$ denote positive constants depending only on the fixed para\-meters $d$, $s$, $u$, $p$, $q$, $v$, $N$, and probably on auxiliary functions. 
Unless otherwise stated their values may vary from line to line. 
With $A \lesssim B$ we mean $ A \leq C B  $ for a constant $ C > 0 $ independent of $A$ and $B$. The notation $ A \sim B $ stands for $A \lesssim B$ and $B \lesssim A$.

\section{Preliminaries}\label{sect:prelim}
\subsection{Triebel-Lizorkin-Morrey Spaces: Definitions and Basic Properties}\label{Sec_Defi}

Triebel-Lizorkin-Morrey spaces $ \mathcal{E}^{s}_{u,p,q}(\R)$ are spaces of distributions built upon Morrey spaces~$\mathcal{M}^{u}_{p}(\R)$. Therefore, at first we recall the definition of the latter. 

\begin{defi}[Morrey space $\mathcal{M}^{u}_{p}(\R)$]
\label{def_mor}
    For $ 0 < p \leq u < \infty$ the Morrey space $\mathcal{M}^{u}_{p}(\R)$ is the collection of all functions $ f \in L_{p}^{\loc}(\R) $ such that
    \begin{align*}
        \norm{f \sep \mathcal{M}^{u}_{p}(\R)}
        := \sup_{y \in \R, r > 0} \abs{B(y,r)}^{\frac{1}{u}-\frac{1}{p}} \left( \int_{B(y,r)} \abs{f(x)}^{p} \d x  \right)^{\frac{1}{p}} < \infty.
    \end{align*} 
\end{defi}

The Morrey spaces $ \mathcal{M}^{u}_{p}(\R)$ are known to be quasi-Banach spaces and Banach spaces if $ p \geq 1$. 
They have many connections to ordinary Lebesgue spaces $ L_{p}(\R)$. 
Indeed, for $ 0 < p_{2} \leq p_{1} \leq u < \infty $ we have
\begin{align*}
    L_{u}(\R) 
    = \mathcal{M}^{u}_{u}(\R) 
    \hookrightarrow \mathcal{M}^{u}_{p_{1}}(\R)
    \hookrightarrow \mathcal{M}^{u}_{p_{2}}(\R).
\end{align*} 
Moreover, in 2005 Tang and Xu proved the following (vector-valued) boundedness of the Hardy-Littlewood maximal operator \textit{\textbf{M}} in Morrey spaces.

\begin{lem}[{\cite[Lemma 2.5]{TangXu}}]\label{l_ineq1}
    Let $ 1 < q \leq \infty$ and $ 1 < p \leq u < \infty    $. Then for all sequences $ \lbrace f_{j} \rbrace_{j = 0 }^{ \infty }  $ of $L_1^\loc(\R)$-functions there holds (with the usual modifications in the case $ q = \infty$)
    \begin{align*}
        \norm{\left( \sum_{j = 0}^{\infty}  \abs{ (\textbf{M}  f_{j})(\cdot) }^{q} \right)^{\frac{1}{q}} \sep  \mathcal{M}^{u}_{p}( \mathbb{R}^d)} 
        \lesssim \norm{ \left( \sum_{j = 0}^{\infty}  \abs{f_{j}(\cdot)}^{q} \right)^{\frac{1}{q}} \sep \mathcal{M}^{u}_{p}( \mathbb{R}^d)}. 
    \end{align*}
\end{lem} 

In order to define $\mathcal{E}^{s}_{u,p,q}(\R)$ we need a so-called smooth dyadic decomposition of unity. Let $\varphi_0 \in C_0^{\infty}({\R})$ be a non-negative function such that $\varphi_0(x) = 1$ if $\abs{x}\leq 1$ and $ \varphi_0 (x) = 0$ if $\abs{x}\geq \frac{3}{2}$. 
For $k\in \N$ we define
$$
    \varphi_k(x) := \varphi_0(2^{-k}x) - \varphi_0(2^{-k+1}x),\qquad\ x \in \R. 
$$
Then 
$$
    \sum_{k=0}^\infty \varphi_k(x) = 1, \qquad x\in \R, 
$$
and
$$
    \supp \varphi_k \subset \left\{x\in \R \sep 2^{k-1}\le \abs{x} \le 2^{k+1}\right\}, \qquad k \in \N,
$$
which justifies the name smooth dyadic decomposition of unity for the system $(\varphi_k)_{k\in \N_0 }$. 
Moreover, using the Paley-Wiener-Schwartz theorem~\cite[Theorem 2 in Section~1.2.1]{Tr83}, we find that for all $f\in\mathcal{S}'(\R)$ the distributions $\cfi[\varphi_{k}\, \cf f]\in\mathcal{S}'(\R)$, $k\in\N_0$, are actually smooth functions on $\R$. 
This allows for the following definition of Triebel-Lizorkin-Morrey spaces.

\begin{defi}[Triebel-Lizorkin-Morrey space $\mathcal{E}^{s}_{u,p,q}(\R)$]
\label{def_tlm}
    Let $ 0 < p \leq u < \infty$, $0 < q \leq \infty$, and $ s \in \mathbb{R}$. Further, let $ (\varphi_{k})_{k\in \N_0 }$ be a smooth dyadic decomposition of unity. Then the Triebel-Lizorkin-Morrey space $  \mathcal{E}^{s}_{u,p,q}(\mathbb{R}^{d})$ collects all $ f \in \mathcal{S}'(\mathbb{R}^{d})$ for which
    \begin{align*} 
        \norm{f \sep \mathcal{E}^{s}_{u,p,q}(\mathbb{R}^{d})} := \norm{ \left( \sum_{k = 0}^{\infty} 2^{ksq} \abs{\mathcal{F}^{-1}[\varphi_{k} \mathcal{F}f](\cdot)}^{q} \right)^{\frac{1}{q}} \sep \mathcal{M}^{u}_{p}(\R)} < \infty.
    \end{align*}
    If $ q = \infty$, the usual modifications are made.
\end{defi}

Let us collect some well-known basic properties of Triebel-Lizorkin-Morrey spaces. Most of them will be used in proofs later on.

\begin{lem}\label{l_bp1}
    Let $ 0 < p \leq u < \infty $, $ 0 < q \leq \infty $ and $ s \in \mathbb{R} $. Then the following holds.
    \begin{enumerate}
        \item $\mathcal{E}^{s}_{u,p,q}(\mathbb{R}^{d})$ is independent of the chosen smooth dyadic decomposition of unity in the sense of equivalent quasi-norms. 
        
        \item The spaces $  \mathcal{E}^{s}_{u,p,q}(\mathbb{R}^{d}) $ are quasi-Banach spaces. For $ p,q \geq 1 $ they are Banach spaces.

        \item With $\tau := \min\{1,p,q\}$  we have
        $$  
            \norm{f + g  \sep \mathcal{E}^{s}_{u,p,q}(\mathbb{R}^{d}) }^{\tau} \leq \norm{f \sep \mathcal{E}^{s}_{u,p,q}(\mathbb{R}^{d}) }^{\tau} + \norm{ g \sep \mathcal{E}^{s}_{u,p,q}(\mathbb{R}^{d}) }^{\tau}, \quad f,g \in \mathcal{E}^{s}_{u,p,q}(\mathbb{R}^{d}).
        $$
        
        \item $\mathcal{S}(\mathbb{R}^{d}) \hookrightarrow    \mathcal{E}^{s}_{u,p,q}(\mathbb{R}^{d}) \hookrightarrow   \mathcal{S}'(\mathbb{R}^{d})$.

        \item $\mathcal{E}^{s}_{p,p,q}(\mathbb{R}^{d}) = F^{s}_{p,q}(\R)$.

        \item If $p>1$, then $\mathcal{E}^{0}_{u,p,2}(\mathbb{R}^{d}) = \mathcal{M}^{u}_{p}(\R)$.
    \end{enumerate}
\end{lem} 

\begin{proof}
    Assertion~(i) was proved in \cite[Theorem 2.8]{TangXu}. The proofs of (ii) and (iii) are standard; we refer to \cite[Lemma 2.1]{ysy}. Also (iv) with a slightly different formulation is proven in \cite[Proposition 2.3]{ysy}. Finally, Assertion~(v) is obvious and (vi) has been shown in \cite[Proposition~4.1]{maz}.  
\end{proof}

As usual, Triebel-Lizorkin-Morrey spaces on domains $\Omega\subsetneq\R$ are defined by restriction.
For this purpose, given some $g\in \mathcal{S}'(\R)\supset \mathcal{E}^{s}_{u,p,q}(\R)$ we let the distribution $g|_\Omega\in \mathcal{D}'(\Omega)$ be defined by $g|_\Omega(\varphi):=g(\varphi)$ for all $\varphi\in \mathcal{D}(\Omega)$ ($\subset \mathcal{D}(\R) \subset \mathcal{S}(\R)$).

\begin{defi}[Triebel-Lizorkin-Morrey space $\mathcal{E}^{s}_{u,p,q}(\Omega)$]
\label{def_tlmD}
    For $ 0 < p \leq u < \infty $, $0 < q \leq \infty$, $s \in \mathbb{R}$, and domains $ \Omega \subsetneq \R$ with $d\in\N$ let
    \begin{align*}
        \mathcal{E}^{s}_{u,p,q}(\Omega) := \left\{ f \in \mathcal{D}'(\Omega) \sep f = g|_\Omega \;\text{ for some } \; g \in       \mathcal{E}^{s}_{u,p,q}(\R)  \right\}
    \end{align*}
    be endowed with the quasi-norm
    \begin{align*}
        \norm{ f \sep \mathcal{E}^{s}_{u,p,q}(\Omega)} := \inf \left\{{\norm{ g \sep \mathcal{E}^{s}_{u,p,q}(\R)}} \sep f = g|_\Omega \; \mbox{ for some } \; g \in       \mathcal{E}^{s}_{u,p,q}(\R) \right\}. 
    \end{align*}
\end{defi}

Later on, we will especially be concerned with $\mathcal{E}^{s}_{u,p,q}(\Omega)$ on Lipschitz domains $\Omega \subset \mathbb{R}^{d}$.
In the last years these spaces attracted some attention in the literature. A Rychkov universal extension operator has been constructed in \cite{ZHS}, see also \cite{GoHaSkr}. Littlewood-Paley type characterizations for those spaces have been obtained recently in \cite{Yao}. 
First results on (complex) interpolation can be found in \cite{ZHS}, see also \cite{YSY2}, and continuous embeddings have been studied in \cite{GoHaSkr,HarSkrNuc}.

For the definition of Lipschitz domains we follow Stein~\cite[VI.3.2]{Stein}.
\begin{defi}[Lipschitz domain]
\label{lipdo}
    Let $d\in\N\setminus\{1\}$.
    \begin{enumerate}
        \item A special Lipschitz domain is an open set $\Omega \subset \mathbb{R}^{d}$ lying above the graph of a Lipschitz function $\omega\colon \mathbb{R}^{d-1} \to \mathbb{R}$, namely,
        $$
            \Omega:= \{(x',x_d) \in \mathbb{R}^{(d-1)+1} \sep x_d > \omega (x')\}.
        $$
        
        \item A bounded Lipschitz domain is a bounded domain $\Omega\subset \mathbb{R}^{d}$ whose boundary $\partial \Omega$ can be covered by a finite number of open balls $B_k$ such that for each $k$ after a suitable rotation $\partial\Omega \cap B_k$ is a part of the graph of a Lipschitz function.
        
        \item By a Lipschitz domain, we mean either a special or a bounded Lipschitz domain. 
    \end{enumerate}
\end{defi}
\begin{rem}
    For notational simplicity we shall use the convention that a (bounded) Lip\-schitz domain in $\mathbb{R}$ is just a (bounded) interval.
\end{rem}

In the spirit of \autoref{def_tlmD} we could also define Morrey spaces on domains. 
Nevertheless, we prefer the following direct approach.
Therein, $M(\Omega)$ denotes the set of all measurable complex-valued functions on $\Omega\subseteq\R$. 
\begin{defi}[Morrey space $\mathcal{M}^{u}_{p}(\Omega)$]
\label{def:morrey_mod}
    For $d\in\N$ let $\Omega\subseteq\R$ be a domain.
    Then for $0<p \leq u < \infty$ the Morrey space $\mathcal{M}^{u}_{p}(\Omega)$ is given by
    $$
        \mathcal{M}^{u}_{p}(\Omega) := \left\{ f\in M(\Omega) \sep \norm{f \sep \mathcal{M}^{u}_{p}(\Omega)}<\infty \right\}
    $$
    endowed with the (quasi-)norm
    $$
        \norm{f \sep \mathcal{M}^{u}_{p}(\Omega)} := \sup_{y\in\Omega, r>0} r^{d(\frac{1}{u}- \frac{1}{p})} \left( \int_{\Omega\cap B(y,r)} \abs{f(x)}^p \d x \right)^{\frac{1}{p}}.
    $$
\end{defi}

\begin{rem}
    Clearly, $\mathcal{M}^{u}_{p}(\Omega)\subseteq L_p^{\loc}(\Omega)$. Furthermore, in Definition \ref{def:morrey_mod} we could use cubes instead of balls and/or replace $r>0$ by $2^j$ with $j\in\mathbb{Z}$ to obtain equivalent quasi-norms. 
    Moreover, let us stress that for $\Omega=\R$ this definition is equivalent to \autoref{def_mor} above. 
    However, a trivial extension of $f\in L_p^{\loc}(\Omega)$ is not granted to be in $ L_p^{\loc}(\R)$. 
    To avoid these difficulties we used $M(\Omega)$ in \autoref{def:morrey_mod}.
    Finally note that the weight does not change for balls close to~$\partial\Omega$ in order to avoid restrictions on $\Omega$ (such as the so-called measure density condition). 
\end{rem}

\subsection{Further Technical Properties}
Let us collect some additional useful properties of the Morrey spaces we just defined. 
\begin{lem}\label{lem:tools_M}
    Let $0< p \leq u < \infty$ and let $\Omega,\Omega_1,\Omega_2\subseteq\R$ be domains.
    \begin{enumerate}
        \item The (restriction of the) trivial extension $E\colon M(\Omega)\to M(\R)$, given by 
        $$
            Ef(x):=\begin{cases}
            f(x), & x\in \Omega,\\
            0, & x\in\R\setminus\Omega,
            \end{cases}
        $$
        provides a linear and continuous extension $\mathcal{M}^{u}_{p}(\Omega) \to \mathcal{M}^{u}_{p}(\R)$.
        
        \item There holds $\mathcal{M}^{u}_{p}(\Omega) = \left\{ f\in M(\Omega) \sep \exists F\in\mathcal{M}^{u}_{p}(\R)\colon F\vert_\Omega=f \text{ a.e.\ in } \Omega \right\}$, 
        where
        \begin{align}\label{eq:normeq}
            \norm{f \sep \mathcal{M}^{u}_{p}(\Omega)} \sim \inf_{\substack{F\in\mathcal{M}^{u}_{p}(\R)\\ F\vert_\Omega = f \text{ a.e.\ in } \Omega}} \norm{F \sep \mathcal{M}^{u}_{p}(\R)}.
        \end{align}
        
        \item For all $G\in L_\infty(\R)$ there exists $c_G>0$ such that
        $$
            \norm{ G\vert_{\Omega} \cdot f \sep \mathcal{M}^{u}_{p}(\Omega)} 
            \leq c_G \norm{ f \sep  \mathcal{M}^{u}_{p}(\Omega)}, \qquad f \in  \mathcal{M}^{u}_{p}(\Omega).
        $$
        
        \item Every affine-linear diffeomorphism $\Phi\colon\R\to\R$ (i.e.\ $\Phi(x)=Ax+b$ with $\abs{\det A}\neq 0$) yields an isomorphism $T_\Phi \colon f\mapsto f\circ \Phi$ of $\mathcal{M}^{u}_{p}(\Phi(\Omega))$ onto $\mathcal{M}^{u}_{p}(\Omega)$.
        
        \item Let $S\subseteq\R$ be such that $S\cap \Omega_1 = S\cap \Omega_2$. Then for all $F\in M(\R)$ with $\supp F\subseteq S$ we have
        $$
            \norm{ F\vert_{\Omega_2} \sep \mathcal{M}^{u}_{p}(\Omega_2)} \sim \norm{ F\vert_{\Omega_1} \sep  \mathcal{M}^{u}_{p}(\Omega_1)}.
        $$
        
        \item If $0<\mu<\infty$, then $f\in \mathcal{M}^{u}_{p}(\Omega)$ if and only if $\abs{f}^\mu \in \mathcal{M}^{u/\mu}_{p/\mu}(\Omega)$. In this case
        $$
            \norm{f \sep \mathcal{M}^{u}_{p}(\Omega)} = \norm{\abs{f}^\mu \sep \mathcal{M}^{\frac{u}{\mu}}_{\frac{p}{\mu}}(\Omega)}^{\frac{1}{\mu}}.
        $$
        
        \item For $F\in \mathcal{M}^{u}_{p}(\R)$, $0< v \leq p$, and $R>0$ we have
        $$
            \norm{ \left( \int_{B(\cdot,R)} \abs{F(x)}^v \d x \right)^{\frac{1}{v}} \sep \mathcal{M}^{u}_{p}(\R)}
            \lesssim \norm{ F \sep \mathcal{M}^{u}_{p}(\R)}.
        $$
    \end{enumerate}
\end{lem}
\begin{proof}
\textit{Step 1. }
    At first we prove (i). W.l.o.g.\ assume $\Omega\neq\R$ and let $f\in \mathcal{M}^{u}_{p}(\Omega)$. Then we clearly have $Ef\in M(\R)$ and $(Ef)\vert_\Omega=f$ (pointwise a.e.) in $\Omega$. We show
    $$
        \norm{Ef \sep \mathcal{M}^{u}_{p}(\R)} \lesssim \norm{f \sep \mathcal{M}^{u}_{p}(\Omega)}.
    $$
    To this end, we adapt the idea of \cite[Theorem~2.3]{HarSchSkr18}
    and let $y\in\R$ as well as $r>0$. 
    If $\distance{y}{\Omega}\geq r$, then $B(y,r)\cap \Omega=\emptyset$ and hence the definition of $E$ yields $\int_{B(y,r)} \abs{Ef(x)}^p \d x = \int_{\Omega \cap B(y,r)} \abs{f(x)}^p \d x=0$.
    If otherwise $\distance{y}{\Omega}< r$, there exists $y'\in\Omega$ such that $B(y,r)\subset B(y',2r)$ and thus 
    $$
        r^{d(\frac{1}{u}- \frac{1}{p})} \left( \int_{B(y,r)} \abs{Ef(x)}^p \d x \right)^{\frac{1}{p}}
        \lesssim (2r)^{d(\frac{1}{u}- \frac{1}{p})} \left( \int_{\Omega\cap B(y',2r)} \abs{f(x)}^p \d x \right)^{\frac{1}{p}}.
    $$
    Therefore,
    \begin{align*}
        \norm{Ef \sep \mathcal{M}^{u}_{p}(\R)} 
        &= \sup_{y\in\R, r>0} r^{d(\frac{1}{u}- \frac{1}{p})} \left( \int_{B(y,r)} \abs{Ef(x)}^p \d x \right)^{\frac{1}{p}} \\
        &\lesssim \sup_{y'\in\Omega, r>0} |B(y',r)|^{\frac{1}{u}- \frac{1}{p}} \left( \int_{\Omega\cap B(y',r)} \abs{f(x)}^p \d x \right)^{\frac{1}{p}}
        \lesssim  \norm{f \sep \mathcal{M}^{u}_{p}(\Omega)}<\infty.
    \end{align*}
This shows (i)  as well as ``$\subseteq$'' in (ii) and ``$\gtrsim$'' in \eqref{eq:normeq}.
    
    \textit{Step 2. } In order to complete the proof of (ii) now let $f\in M(\Omega)$ and $F\in\mathcal{M}^{u}_{p}(\R)$ with $F\vert_\Omega=f$ a.e.\ in $\Omega$.  Then we observe
    \begin{align*}
        \norm{f \sep \mathcal{M}^{u}_{p}(\Omega)} 
        & \lesssim \sup_{y\in\Omega, r>0} r^{d(\frac{1}{u}- \frac{1}{p})} \left( \int_{\Omega\cap B(y,r)} \abs{F(x)}^p \d x \right)^{ \frac{1}{p}} \\
        &\leq \sup_{y\in\R, r>0} r^{d(\frac{1}{u}- \frac{1}{p})} \left( \int_{B(y,r)} \abs{F(x)}^p \d x \right)^{ \frac{1}{p}}
        =\norm{F \sep \mathcal{M}^{u}_{p}(\R)}<\infty
    \end{align*}
    which shows $f\in \mathcal{M}^{u}_{p}(\Omega)$, i.e.\ ``$\supseteq$'' in~(ii).
    Since $F$ was arbitrary, we further have ``$\lesssim$'' in~\eqref{eq:normeq} and hence (ii).
    
    \textit{Step 3. }
    Since (iii) directly follows from \autoref{def:morrey_mod}, we move on and show (iv).
    For $f\in \mathcal{M}^{u}_{p}(\Phi(\Omega))\subset M(\Phi(\Omega))$ we have $f\circ \Phi \in M(\Omega)$.
    Moreover, for $y\in\Omega$ and $r>0$ a transformation of measure shows
    $$
        \int_{\Omega\cap B(y,r)} \abs{f(\Phi(x))}^p \d x = \frac{1}{\abs{\det A}} \int_{\Phi(\Omega\cap B(y,r))} \abs{f(z)}^p \d z,
    $$
    where $\Phi(\Omega\cap B(y,r)) \subseteq \Phi(\Omega)\cap \Phi(B(y,r)) \subseteq \Phi(\Omega)\cap B(\Phi(y),\norm{A}_2 r)$ since
    $$
        \abs{z - \Phi(y)} = \abs{\Phi(x)-\Phi(y)} = \abs{A\,(x-y)} \leq \norm{A}_2 \abs{x-y}
    $$
    for all $z= \Phi(x) \in \Phi(B(y,r))$, where $\norm{A}_2\neq 0$ denotes the spectral norm of the matrix~$A$.
    Thus 
    \begin{align*}
        \norm{f\circ\Phi \sep \mathcal{M}^{u}_{p}(\Omega)}
        &\lesssim \sup_{y\in\Omega, r>0} r^{d(\frac{1}{u}- \frac{1}{p})} \left( \int_{\Phi(\Omega)\cap B(\Phi(y),\norm{A}_2 r)} \abs{f(z)}^p \d z \right)^{ \frac{1}{p}} \\
        &\lesssim \sup_{x\in\Phi(\Omega), R>0} R^{d(\frac{1}{u}- \frac{1}{p})} \left( \int_{\Phi(\Omega)\cap B(x,R)} \abs{f(z)}^p \d z \right)^{ \frac{1}{p}} = \norm{f \sep \mathcal{M}^{u}_{p}(\Phi(\Omega))}
    \end{align*}
    which also shows $f\circ\Phi \in \mathcal{M}^{u}_{p}(\Omega)$.
    
    \textit{Step 4. }
    To show (v) it suffices to prove a one-sided estimate as the other one follows by exchanging the roles of $\Omega_1$ and $\Omega_2$. Using (i), the support properties of $F$ imply
    \begin{align*}
        \norm{F\vert_{\Omega_2} \sep \mathcal{M}^{u}_{p}(\Omega_2)}
        &\lesssim \sup_{y\in\Omega_2, r>0} r^{d(\frac{1}{u}- \frac{1}{p})} \left( \int_{S\cap \Omega_1\cap B(y,r)} \abs{F(x)}^p \d x \right)^{ \frac{1}{p}} \\
        &\leq \sup_{y\in\R, r>0} r^{d(\frac{1}{u}- \frac{1}{p})} \left( \int_{\R} \abs{[E(F\vert_{\Omega_1})](x)}^p \d x \right)^{\frac{1}{p}} \\
        &=\norm{ E(F\vert_{\Omega_1}) \sep \mathcal{M}^{u}_{p}(\R)}\\
        &\lesssim \norm{F\vert_{\Omega_1} \sep \mathcal{M}^{u}_{p}(\Omega_1)}.
    \end{align*}
    This completes the proof of (v).
    
    \textit{Step 5. }
    Since (vi) directly follows from \autoref{def:morrey_mod}, it remains to show (vii). We have
    \begin{align*}
        L &:= \norm{ \left( \int_{B(\cdot,R)} \abs{F(x)}^v \d x \right)^{\frac{1}{v}} \sep \mathcal{M}^{u}_{p}(\R)} 
        = \norm{ \left( \int_{B(0,R)} \abs{F(\cdot-h)}^v \d h \right)^{\frac{1}{v}} \sep \mathcal{M}^{u}_{p}(\R)}.
    \end{align*}
    Then, since $v\leq p$, Minkowski's integral inequality yields
    \begin{align*}
        L &= \sup_{y\in\R, r>0} r^{d(\frac{1}{u}- \frac{1}{p})} \left( \int_{B(y,r)}  \left( \int_{B(0,R)} \abs{F(x-h)}^v \d h \right)^{\frac{p}{v}} \d x \right)^{\frac{1}{p}} \\
        &\leq \sup_{y\in\R, r>0} r^{d(\frac{1}{u}- \frac{1}{p})} \left( \int_{B(0,R)}  \left( \int_{B(y,r)} \abs{F(x-h)}^p \d x \right)^{\frac{v}{p}} \d h \right)^{\frac{1}{v}}        
    \end{align*}
    such that
    \begin{align*}
        L &\leq \left( \int_{B(0,R)} \left[ \sup_{y\in\R, r>0} r^{d(\frac{1}{u}- \frac{1}{p})} \left( \int_{B(y,r)} \abs{F(x-h)}^p \d x \right)^{\frac{1}{p}} \right]^v \d h \right)^{\frac{1}{v}} \\
        &= \left( \int_{B(0,R)} \norm{ F(\cdot-h) \sep \mathcal{M}^{u}_{p}(\R) }^v \d h \right)^{\frac{1}{v}}
    \end{align*}
    and hence the translation invariance of $\mathcal{M}^{u}_{p}(\R)$ completes the proof.
\end{proof}

Later also the following observation will be important.

\begin{lem}\label{lem:avg}
    For $d\in\N$ let $\Omega\subseteq\R$ be a domain or $\R$ itself. Further let $0<p \leq u <\infty$, $0<v \leq \infty$ and $0 < r \leq R < \infty$. 
    Then for $f\in L_{\max\{1,p,v\}}^\loc(\Omega)$ we have
    $$
        \norm{\left(\int_{B(\cdot,R)\cap \Omega} \abs{f(y)}^v\d y\right)^{\frac{1}{v}} \sep \mathcal{M}^{u}_{p}(\Omega)}
        \sim \norm{\left(\int_{B(\cdot,r)\cap \Omega} \abs{f(y)}^v\d y\right)^{\frac{1}{v}} \sep \mathcal{M}^{u}_{p}(\Omega)}
    $$
    with constants independent of $f$.
\end{lem}
\begin{proof}
    Due to the monotonicity of the integral in $r$ it suffices to prove ``$\lesssim$''.
    In addition, \autoref{lem:tools_M}(vi) shows that if $v<\infty$, we can restrict ourselves to $v=1$. (If $v=\infty$, then one has to replace every integral by an essential supremum in the subsequent proof.) 
    Further note that there exist displacement vectors $w_k\in \Z$, $k\in\N$, such that
    for some $K=K(r,R,d)\in\N$,
    $$
        B(x,R) \subset \bigcup_{k=1}^{K} B\!\left(x+ \frac{r}{2} w_k , \frac{r}{2} \right),
        \qquad x\in\R.
    $$
    Therefore, \autoref{lem:tools_M}(ii) yields
    \begin{align}
        \norm{\int_{B(\cdot,R)\cap \Omega} \abs{f(y)}\d y \sep \mathcal{M}^{u}_{p}(\Omega)}
        &\sim \norm{\chi_\Omega \int_{B(\cdot,R)} \abs{Ef(y)}\d y \sep \mathcal{M}^{u}_{p}(\R)} \label{eq:proof_normswitch}\\
        &\leq \norm{\sum_{k=1}^K \chi_\Omega \int_{B\left(\,\cdot\,+ \frac{r}{2} w_k , \frac{r}{2} \right)} \abs{Ef(y)}\d y \sep \mathcal{M}^{u}_{p}(\R)}. \nonumber
    \end{align}
    Now, given $x\in\R$ and $k\in\N$, we have
    \begin{align*}
        &\chi_\Omega \int_{B(x+ \frac{r}{2}  w_k , \frac{r}{2} )} \abs{Ef(y)}\d y \\
        &\qquad\leq \sum_{e\in\{-1,0,1\}^d} \chi_\Omega(x) \, \chi_{\Omega}\!\left(x+ \frac{r}{2} [w_k+e] \right) \int_{B\left(x+ \frac{r}{2}  [w_k+e], r\right)} \abs{Ef(y)}\d y
    \end{align*}
    since the left-hand side is strictly positive only if $x\in\Omega$ and at least for one $e^\ast\in\{-1,0,1\}^d$ also $x+ r/2 [w_k+e^\ast]  \in \Omega$ (as $\supp Ef \subseteq \Omega$). 
    For this~$e^\ast$ we have 
    $$ 
        B\!\left(x+ \frac{r}{2}  w_k , \frac{r}{2} \right) \subset B\!\left(x+ \frac{r}{2} [w_k+e^\ast], r\right)
    $$ 
    Hence,
    \begin{align*}
        &\norm{\int_{B(\cdot,R)\cap \Omega} \abs{f(y)}\d y \sep \mathcal{M}^{u}_{p}(\Omega)} \\
        &\quad\lesssim \norm{\sum_{k=1}^K \sum_{e\in\{-1,0,1\}^d} \chi_\Omega(\,\cdot\,) \, \chi_{\Omega}\!\left(\,\cdot\,+ \frac{r}{2} [w_k+e] \right) \int_{B\left(\,\cdot\,+ \frac{r}{2}  [w_k+e], r \right)} \abs{Ef(y)}\d y \sep \mathcal{M}^{u}_{p}(\R)}\\
        &\quad\lesssim \sum_{k=1}^K \sum_{e\in\{-1,0,1\}^d} \norm{\chi_{\Omega}\!\left(\,\cdot\,+ \frac{r}{2} [w_k+e] \right) \int_{B\left(\,\cdot\,+  \frac{r}{2} [w_k+e], r \right)} \abs{Ef(y)}\d y \sep \mathcal{M}^{u}_{p}(\R)} \\
        &\quad\sim \norm{\chi_{\Omega}(\,\cdot\,) \int_{B(\,\cdot\,,r)} \abs{Ef(y)}\d y \sep \mathcal{M}^{u}_{p}(\R)},
    \end{align*}
    where for the last step we used the translation-invariance of $\mathcal{M}^{u}_{p}(\R)$. 
    Rewriting the last expression similiar to \eqref{eq:proof_normswitch} completes the proof.
\end{proof}

We complement some of the assertions of \autoref{lem:tools_M} by corresponding results for Triebel-Lizorkin-Morrey spaces.
For $S\subseteq\R$ and $\varepsilon>0$ we let $S_\varepsilon:=\{x\in\R \sep \distance{x}{S}<\varepsilon\}$ denote its $\varepsilon$-neighbourhood.
\begin{lem}\label{lem:tools_E}
    Let $0<p \leq u < \infty$, $0< q \leq \infty$, $s\in\re$, and let $\Omega,\Omega_1,\Omega_2\subseteq\R$ be  domains.
    \begin{enumerate}
        \item For all $G \in \mathcal{D}(\R)$ there exists $c_G>0$  such that
        $$
            \norm{G\vert_\Omega\cdot f \sep \mathcal{E}^{s}_{u,p,q}(\Omega)} 
            \leq c_G \norm{f \sep \mathcal{E}^{s}_{u,p,q}(\Omega)}, \qquad f\in \mathcal{E}^{s}_{u,p,q}(\Omega).
        $$

        \item Every affine-linear diffeomorphism $\Phi\colon\R\to\R$ yields an isomorphism $T_\Phi \colon f\mapsto f\circ \Phi$ of $\mathcal{E}^{s}_{u,p,q}(\Phi(\Omega))$ onto $\mathcal{E}^{s}_{u,p,q}(\Omega)$.
    
        \item Let $S\subseteq\R$ be such that $S_\varepsilon\cap \Omega_1= S_\varepsilon\cap \Omega_2$ . Then for all $F\in \mathcal{S}'(\R)$ with $\supp F\subseteq S$
        $$
            \norm{F\vert_{\Omega_1} \sep \mathcal{E}^{s}_{u,p,q}(\Omega_1)}
            \sim \norm{F\vert_{\Omega_2} \sep \mathcal{E}^{s}_{u,p,q}(\Omega_2)}.
        $$
        
        \item If $s>\sigma_p $, then $ \mathcal{E}^{s}_{u,p,q}(\Omega) \hookrightarrow \mathcal{M}^u_p(\Omega)$.
    \end{enumerate}
\end{lem}
\begin{proof}
\textit{Step 1. }
    At first we prove (i). Again, w.l.o.g., we may assume $\Omega\neq \R$.
    For given $f\in \mathcal{E}^{s}_{u,p,q}(\Omega)$ we choose an extension $F\in \mathcal{E}^{s}_{u,p,q}(\R)$ with
    $$
        \norm{F \sep \mathcal{E}^{s}_{u,p,q}(\R)} \leq 2 \norm{f \sep \mathcal{E}^{s}_{u,p,q}(\Omega)}.
    $$
    Then $(G F)\vert_\Omega(\varphi)=(GF)(E\varphi)=F(G E\varphi)=f(G\vert_\Omega \varphi)=(G\vert_\Omega f)(\varphi)$ for all $\varphi\in\mathcal{D}(\Omega)$, i.e., $(G F)\vert_\Omega=G\vert_\Omega f$ in $\mathcal{D}'(\Omega)$.
    Therefore, with constants independent of $f$,
    \begin{align*}
        \norm{G\vert_\Omega\cdot f \sep \mathcal{E}^{s}_{u,p,q}(\Omega)} \leq \norm{G\cdot F \sep \mathcal{E}^{s}_{u,p,q}(\R)} \leq C_G \norm{F \sep \mathcal{E}^{s}_{u,p,q}(\R)} \leq 2\, C_G \norm{f \sep \mathcal{E}^{s}_{u,p,q}(\Omega)},
    \end{align*}
    where we used \cite[Theorem~1.5]{Saw10}, see also \cite[Theorem 6.1]{ysy}, for the second estimate.
    
    \textit{Step 2. }
    We show (ii). 
    Using \cite[Theorem~1.7]{Saw10} or \cite[Theorem 6.7]{ysy} we find that $T_\Phi\colon F\mapsto F\circ\Phi$ is an isomorphism on $\mathcal{E}^{s}_{u,p,q}(\R)$. Hence, 
    $F \in \mathcal{E}^{s}_{u,p,q}(\R)$ is equivalent to $G:=F\circ\Phi^{-1} \in \mathcal{E}^{s}_{u,p,q}(\R)$.
    Further, for $f \in \mathcal{E}^{s}_{u,p,q}(\Phi(\Omega))$ as well as $F \in \mathcal{E}^{s}_{u,p,q}(\R)$ we have
    $$
        (f\circ\Phi)(\varphi) = f\left(\frac{\varphi\circ \Phi^{-1}}{\abs{\det A}}\right)
        \quad\text{and}\quad
        F\vert_\Omega (\varphi) = [(F \circ \Phi^{-1})\vert_{\Phi(\Omega)}\circ \Phi)](\varphi) = G\vert_{\Phi(\Omega)}\left(\frac{\varphi\circ \Phi^{-1}}{\abs{\det A}}\right)
    $$
    for all $\varphi\in\mathcal{D}(\Omega)$ such that $f\circ\Phi=F\vert_\Omega$ in $\mathcal{D}'(\Omega)$ is equivalent to $f=G\vert_{\Phi(\Omega)}$ in $\mathcal{D}'(\Phi(\Omega))$.
    Therefore, \cite[Theorem~1.7]{Saw10} implies
    \begin{align*}
        \norm{f\circ\Phi \sep \mathcal{E}^{s}_{u,p,q}(\Omega)} 
        &= \inf_{\substack{F\in \mathcal{E}^{s}_{u,p,q}(\R)\\ f\circ\Phi=F\vert_\Omega \text{ in } \mathcal{D}'(\Omega)}} \norm{F \sep \mathcal{E}^{s}_{u,p,q}(\R)} \\
        &= \inf_{\substack{G\in \mathcal{E}^{s}_{u,p,q}(\R)\\ f=G\vert_{\Phi(\Omega)} \text{ in } \mathcal{D}'(\Phi(\Omega))}} \underbrace{\norm{G\circ\Phi \sep \mathcal{E}^{s}_{u,p,q}(\R)}}_{\lesssim \norm{G \sep \mathcal{E}^{s}_{u,p,q}(\R)}} 
        \lesssim \norm{f \sep \mathcal{E}^{s}_{u,p,q}(\Phi(\Omega))}.
    \end{align*}
    
    \textit{Step 3. }
    In order to prove (iii), let $\sigma\in\mathcal{D}(\R)$ be such that $\sigma\equiv 1$ on $S$ with $\supp \sigma\subset S_\varepsilon$. 
    Then for all $F\in\mathcal{S}'(\R)$ with $\supp F\subseteq S$, we have $F=\sigma F$ in $\mathcal{S}'(\R)$.
    Now assume $F\vert_{\Omega_2}\in \mathcal{E}^{s}_{u,p,q}(\Omega_2)$ and let $G\in \mathcal{E}^{s}_{u,p,q}(\R)$ with $F\vert_{\Omega_2}=G\vert_{\Omega_2}$ in $\mathcal{D}'(\Omega_2)$ be arbitrarily fixed.
    Then (i) shows that $\widetilde{G}:=\sigma G \in \mathcal{E}^{s}_{u,p,q}(\R)$. 
    If $E_j\colon \mathcal{D}(\Omega_j)\to\mathcal{S}(\R)$, $j=1,2$, denote the trivial extensions, then for all $\varphi\in\mathcal{D}(\Omega_1)$ we have $[\sigma (E_1\varphi)]\vert_{\Omega_2}\in \mathcal{D}(\Omega_2)$,
    $$
        G\vert_{\Omega_2}([\sigma (E_1\varphi)]\vert_{\Omega_2})
        = G(E_2[\sigma (E_1\varphi)]\vert_{\Omega_2})
        = G(\sigma E_1\varphi)
        = (\sigma G)(E_1\varphi) 
        = \widetilde{G}\vert_{\Omega_1}(\varphi)
    $$
    and similarly $F\vert_{\Omega_2}([\sigma (E_1\varphi)]\vert_{\Omega_2}) = (\sigma F)(E_1\varphi) = F(E_1\varphi) = F\vert_{\Omega_1}(\varphi)$.
    So we conclude that $F\vert_{\Omega_1} = \widetilde{G}\vert_{\Omega_1}$ in $\mathcal{D}'(\Omega_1)$ and hence using (i) with $\Omega:=\R$ we find
    \begin{align*}
        \norm{F\vert_{\Omega_1} \sep \mathcal{E}^{s}_{u,p,q}(\Omega_1)}  \leq \norm{\widetilde{G} \sep \mathcal{E}^{s}_{u,p,q}(\R)} 
        = \norm{\sigma G \sep \mathcal{E}^{s}_{u,p,q}(\R)} 
        \leq c_\sigma \norm{G \sep \mathcal{E}^{s}_{u,p,q}(\R)}.
    \end{align*}
    Since we have $G\in \mathcal{E}^{s}_{u,p,q}(\R)$ with $F\vert_{\Omega_2}=G\vert_{\Omega_2}$ in $\mathcal{D}'(\Omega_2)$ was arbitrary, this shows $\norm{F\vert_{\Omega_1} \sep \mathcal{E}^{s}_{u,p,q}(\Omega_1)}\lesssim \norm{F\vert_{\Omega_2} \sep \mathcal{E}^{s}_{u,p,q}(\Omega_2)}$
    and the proof of (iii) is complete.
    
    \textit{Step 4. } It remains to prove (iv). 
    For this purpose, we first discuss $\Omega=\R$ and distinguish two cases. For $p>1$ the assumption $s>\sigma_p$ implies $s>0$ and hence
    $$
        \mathcal{E}^{s}_{u,p,q}(\R) 
        \hookrightarrow \mathcal{E}^{0}_{u,p,2}(\R) 
        = \mathcal{M}^u_p(\R),
    $$
    due to \autoref{l_bp1}(vi). Here we also could refer to \cite[Theorem 3.2]{HaMoSk2}. 
    
    If otherwise $p \leq 1$, at first we recall that $\mathcal{E}^{s}_{u,p,q}(\R) \hookrightarrow \mathcal{E}^{s}_{u,p,\infty}(\R)$. Since $s>\sigma_p$ implies $s>\sigma_p \cdot p/u$ the latter space is contained in $L_1^\loc(\R)$ and hence in $L_{p}^\loc(\R)$, see \cite[Theorem 3.3]{HaMoSk}. 
    Thus, \cite[Theorem 1.2]{Ho1} implies $\mathcal{E}^{s}_{u,p,\infty}(\R) \hookrightarrow \mathcal{M}^u_p(\R)$. Notice that also here we can refer to \cite[Theorem 3.2]{HaMoSk2}. 
    
    Now let $\Omega\subsetneq \R$. Then \autoref{lem:tools_M}(ii) and the previous findings yield
    \begin{align*}
        \norm{f \sep \mathcal{M}^{u}_{p}(\Omega)}  \lesssim  \inf_{\substack{F\in\mathcal{E}^{s}_{u,p,q}(\R)\\ F\vert_\Omega = f \text{ in } \mathcal{D}'(\Omega)}} \norm{F \sep \mathcal{E}^{s}_{u,p,q}(\R)} = \norm{f \sep \mathcal{E}^{s}_{u,p,q}(\Omega)},
    \end{align*}
    where we used that, due to \cite[Theorem 3.3]{HaMoSk}, $\mathcal{E}^{s}_{u,p,q}(\R)\subset L_1^\loc(\R)$ such that $F|_\Omega$ is regular and its a.e.\ pointwise equality with $f$ is equivalent to equality in the sense of~$\mathcal{D}'(\Omega)$.
\end{proof}

\subsection{Quasi-norms of Interest}\label{sect:quasi-norms}
In this paper, we aim at measuring smoothness of functions in terms of quasi-norms based on higher order differences and local oscillations.

Let $d\in\N$. Given a function $ f \colon \mathbb{R}^d \rightarrow \mathbb{C}$ its first order difference of step length $h\in\R$ is given by the mapping $x\mapsto \Delta_{h}^{1}f (x) := f ( x + h ) - f (x)$ on $\R$. Moreover, for $N\in\N\setminus\{1\}$ we let
\begin{align}\label{eq_diff_def}
    \Delta_{h}^{N}f (x) := \left  (\Delta_{h} ^1 \left ( \Delta_{h} ^{N-1}f \right  )\right ) (x) \, , \qquad x \in \R\, . 
\end{align}
Then it is easily seen that for every $N\in\N$ and $h\in\R$ there holds
\begin{align}\label{eq:Delta}
    \Delta_{h}^{N}f (x) = \sum_{k=0}^N (-1)^{N-k} \binom{N}{k} \,f(x+kh), \qquad x\in\R.
\end{align}
In addition, for $N\in\N$ let $\mathcal{P}_{N-1}$ denote the set of all $d$-variate polynomials with total degree strictly less than $N$ and let $\Omega \subseteq \R$ be a domain or $\R$ itself.
Then for $0<v\leq\infty$ the local oscillation with radius $t>0$ of $f \in L_v^{\loc}(\Omega)$ in $x\in\Omega$ is given by
\begin{align*}
    \osc_{v,\Omega}^{N-1}f(x,t) 
    := \begin{cases}
        \displaystyle \inf_{P \in \mathcal{P}_{N-1}} \Big( t^{-d} \int_{B(x,t) \cap \Omega} \abs{f(y) - P(y)}^{v} \d y \Big)^{\frac{1}{v}} & \text{if}\quad v<\infty,\\
        \displaystyle\inf_{P\in\mathcal{P}_{N-1}} \esssup_{y\in B(x,t)\cap \Omega} \abs{f(y)-P(y)} & \text{if}\quad v=\infty.
    \end{cases}
\end{align*}
If $\Omega=\R$, we simply write $\osc_{v}^{N-1}f:=\osc_{v,\R}^{N-1}f$.

Now we are well-prepared to introduce the quasi-(semi-)norms which will turn out to characterize Triebel-Lizorkin-Morrey spaces $\mathcal{E}^{s}_{u,p,q}(\Omega)$. They generalize respective quantities for Triebel-Lizorkin spaces $F^s_{p,q}(\Omega)$ in a natural way; cf.~\cite[Section 5.2]{Tr92}.
\begin{defi}[New quasi-(semi-)norms on $\mathcal{E}^{s}_{u,p,q}$]
\label{defi:norms}
    For $d\in\N$ let $\Omega \subseteq \R$ be a domain or~$\R$ itself.
    Then for $N\in\N$, $0<p\leq u < \infty$, $0 < q,T,v \leq \infty$, $0<R<\infty$, and $f\in L_{\max\{1,p,v\}}^\loc(\Omega)$ we let
    \begin{align*}
        \abs{f}^{(T,v,N)}_{\osc,\Omega} 
        := \norm{\bigg( \int_{0}^{T} \big[ t^{-s} \, \osc_{v,\Omega}^{N-1} f(\cdot,t) \big]^{q} \frac{\d t}{t} \bigg)^{\frac{1}{q}} \sep \mathcal{M}^{u}_{p}(\Omega)}
    \end{align*}
    (if $\Omega=\R$, we simply write $\abs{f}^{(T,v,N)}_{\osc}:=\abs{f}^{(T,v,N)}_{\osc,\R}$),
    as well as
    \begin{align*}
	   \norm{f \sep \mathcal{E}^{s}_{u,p,q}(\Omega)}^{(T,v,N)}_\osc 
	   &:= \norm{ f \sep \mathcal{M}^{u}_{p}(\Omega)} + \abs{f}^{(T,v,N)}_{\osc,\Omega},\\
	   \norm{f \sep \mathcal{E}^{s}_{u,p,q}(\Omega)}^{(R,T,v,N)}_\osc
	   &:= \norm{\Big( \int_{B(\,\cdot\,,R)\cap \Omega} \abs{f(y)}^{v} \d y \Big )^{\frac{1}{v}} \sep \mathcal{M}^{u}_{p}(\Omega)} + \abs{f}^{(T,v,N)}_{\osc,\Omega}.
    \end{align*}
    Likewise, using $V^{N}(x,t) := \{ h \in \mathbb{R}^d \sep \abs{h} < t \ \text{and} \ x + \ell h \in \Omega \ \text{for} \ 0 \leq \ell \leq N \}$, we let
    \begin{align*}
        \abs{f}^{(T,v,N)}_{\Delta,\Omega} 
        := \norm{\bigg( \int_{0}^{T} \bigg[ t^{-s} \Big( t^{-d} \int_{V^{N}(\,\cdot\,,t)} \abs{\Delta^{N}_{h}f(\cdot)}^v \d h \Big)^{\frac{1}{v}} \bigg]^q \frac{\d t}{t} \bigg)^{\frac{1}{q}} \sep \mathcal{M}^{u}_{p}(\Omega) },
    \end{align*}
    ($\abs{f}^{(T,v,N)}_{\Delta}:=\abs{f}^{(T,v,N)}_{\Delta,\R}$)
    and define $\norm{f \sep \mathcal{E}^{s}_{u,p,q}(\Omega)}^{(T,v,N)}_\Delta$ as well as $\norm{f \sep \mathcal{E}^{s}_{u,p,q}(\Omega)}^{(R,T,v,N)}_\Delta$ correspondingly.
    In any case, for $ q = \infty $ and/or $ v = \infty$, the usual modifications are made.
\end{defi}

\section{Characterizations of \texorpdfstring{$\mathcal{E}^{s}_{u,p,q}(\R)$}{Esupq(Rd)}}
\label{sect:characterizations_Rd}
In this section we prove new characterizations of Triebel-Lizorkin-Morrey spaces~$\mathcal{E}^{s}_{u,p,q}(\R)$ on $\R$ in terms of higher order differences and local oscillations.

\subsection{Differences on \texorpdfstring{$\R$}{Rd}}\label{sect:diff_Rd}
Note that for the spaces $\mathcal{E}^{s}_{u,p,q}(\mathbb{R}^d)$ defined on $\mathbb{R}^d$ characterizations in terms of differences already exist. 
Indeed, the following result is a simple consequence of \cite[Theorem 4.2]{Ho1}.

\begin{satz}\label{thm:diff_Rd}
    Let $d,N \in \mathbb{N}$, $0< p \leq u < \infty$, $0 < q,v \leq \infty$, $0<R<\infty$, and $s\in\mathbb{R}$ with
    \begin{align*}
		d\, \max \left\{  0, \frac{1}{p} - 1, \frac{1}{q} - 1, \frac{1}{p} - \frac{1}{v}, \frac{1}{q} - \frac{1}{v} \right\} < s < N.
	\end{align*}
	Then for all $1\leq T \leq \infty$
	\begin{align}
        \mathcal{E}^{s}_{u,p,q}(\R) 
        &= \left\{ f\in L_{\max\{1,p,v\}}^\loc(\R) \sep \norm{f \sep \mathcal{E}^{s}_{u,p,q}(\R)}^{(T,v,N)}_\Delta < \infty \right\} \label{eq:diff_Rd}\\ 
        &= \left\{ f\in L_{\max\{1,p,v\}}^\loc(\R) \sep \norm{f \sep \mathcal{E}^{s}_{u,p,q}(\R)}^{(R,T,v,N)}_\Delta < \infty \right\} \nonumber
    \end{align}
    and the quasi-norms $\norm{\,\cdot \sep \mathcal{E}^{s}_{u,p,q}(\mathbb{R}^{d})}$, $\norm{\,\cdot \sep \mathcal{E}^{s}_{u,p,q}(\R)}^{(T,v,N)}_\Delta$ and $\norm{\,\cdot \sep \mathcal{E}^{s}_{u,p,q}(\R)}^{(R,T,v,N)}_\Delta$ are mutually equivalent on $L_{\max\{1,p,v\}}^\loc(\R)$. 
\end{satz}
Note that, in particular, \autoref{thm:diff_Rd} yields part (i) of our main \autoref{thm_main_2}.
\begin{proof}
    \emph{Step 1. } In \cite[Theorem 4.2]{Ho1} it has been shown that a function $f \in L_{\min\{p,q\}}^\loc(\R)$ belongs to $\mathcal{E}^{s}_{u,p,q}(\R)$ if and only if $f \in L_v^\loc(\R)$ and
    $\norm{f \sep \mathcal{E}^{s}_{u,p,q}(\R)}^{(T,v,N)}_\Delta$ is finite with $\norm{\,\cdot \sep \mathcal{E}^{s}_{u,p,q}(\R)}^{(T,v,N)}_\Delta$ being equivalent to $\norm{\,\cdot \sep \mathcal{E}^{s}_{u,p,q}(\R)}$ on $L_{\min\{p,q\}}^\loc(\R)$.
    Since $L_{\max\{1,p,v\}}^\loc(\R) \subset L_{\min\{p,q\}}^\loc(\R)$ it thus suffices to prove \eqref{eq:diff_Rd}.
    
    \emph{Step 2. }
    We show \eqref{eq:diff_Rd}. To this end, note that $f\in L_{\max\{1,p,v\}}^\loc(\R)$ implies that $f$ belongs to $L_{\min\{p,q\}}^\loc(\R) \cap L_v^\loc(\R)$. 
    Hence, if $\norm{f \sep \mathcal{E}^{s}_{u,p,q}(\R)}^{(T,v,N)}_\Delta < \infty$, then Step~1 yields $f\in \mathcal{E}^{s}_{u,p,q}(\R)$ which proves ``$\supseteq$'' in~\eqref{eq:diff_Rd}.
    
    For the converse direction, we let $f\in \mathcal{E}^{s}_{u,p,q}(\R)$.
    If $p\leq 1$, then our assumptions imply that $s > \sigma_p \geq \sigma_p \, \frac{p}{u}$ which in turn yields $f\in L_1^\loc(\R)$, see \cite[Theorem 3.3]{HaMoSk}.
    Therefore, we have $f\in L_{\min\{p,q\}}^\loc(\R)$ such that Step~1 implies $f\in L_v^\loc(\R)$ and that $\norm{f \sep \mathcal{E}^{s}_{u,p,q}(\R)}^{(T,v,N)}_\Delta$ is finite. 
    Since $\max\{1,p,v\}=\max\{1,v\}$ if $p\leq 1$, we have shown~``$\subseteq$'' in \eqref{eq:diff_Rd} for this case.
    If otherwise $p>1$, then \autoref{lem:tools_E}(iv) implies 
    $$
        \mathcal{E}^{s}_{u,p,q}(\R) \hookrightarrow \mathcal{M}^u_p(\R) 
        \subset L_p^\loc(\R) \subseteq L_{\min\{p,q\}}^\loc(\R)
    $$
    such that we can apply Step~1 to conclude $f\in L_v^\loc(\R)$ and $\norm{f \sep \mathcal{E}^{s}_{u,p,q}(\R)}^{(T,v,N)}_\Delta < \infty$.
    As $\max\{1,p,v\}=\max\{p,v\}$ for $p> 1$, the proof of ``$\subseteq$'' in \eqref{eq:diff_Rd} is complete.
    
    \emph{Step 3. } In order to complete the proof, we argue similarly based on a modified version of \cite[Theorem 4.2]{Ho1} with $\norm{\big( \int_{B(\,\cdot\,,R)} \abs{f(y)}^{v} \d y \big )^{\frac{1}{v}} \sep \mathcal{M}^{u}_{p}(\R)}$ in place of $\norm{f \sep \mathcal{M}^u_p(\R)}$ which in turn requires a corresponding modification of \cite[Proposition 4.1]{Ho1}.
    However, as this can be done easily we omit the details.
\end{proof}

\begin{rem}\label{hist_rem_diff1}
    In the literature there exist several modified versions of \cite[Theorem 4.2]{Ho1} and hence of \autoref{thm:diff_Rd}. 
    Here we want to refer to \cite[Section 4.3.1]{ysy} and \cite[Section 3.9.2]{Si1} at least. 
    A characterization for Triebel-Lizorkin-Morrey spaces in terms of differences using more restrictive conditions on the parameters also can be found in \cite[Theorem 4.1]{Dri1}.
\end{rem}

Results such as \autoref{thm:diff_Rd} sometimes are called $v$-mean characterizations of higher order differences. In the special case $v=1$ we obtain so-called ball mean characterizations. Moreover the choice $v=q$ can be used to prove Stein-Strichartz characterizations. Much more details concerning this topic can be found in \cite{Ho1} and \cite{H21}. 
There also a discussion concerning the necessity of the conditions on the parameters in \autoref{thm:diff_Rd} can be found.

\subsection{Oscillations on \texorpdfstring{$\R$}{Rd}}\label{re_sec_osc_Rd}
Here our main goal is to describe Triebel-Lizorkin-Morrey spaces $\mathcal{E}^{s}_{u,p,q}(\mathbb{R}^{d})$ in terms of local oscillations. 
In the literature there already exist some first characterizations of this type, e.g., based on the theory developed by Hedberg and Netrusov~\cite{HN}.
Similarly, we obtain the following result.

\begin{satz}\label{thm:osc_clubsuit}
    Let $d,N \in \mathbb{N}$, $s\in\mathbb{R}$, $0< p \leq u < \infty$, and $0 < q,v \leq \infty$ such that
	\begin{align*}
		d\, \max\! \left\{  0 , \frac{1}{p} - 1 , \frac{1}{q} - 1  ,  \frac{1}{p} - \frac{1}{v} ,  \frac{1}{q} - \frac{1}{v}  \right\}  < s < N.
	\end{align*}
    Then
    \begin{align*}
        \mathcal{E}^{s}_{u,p,q}(\R) 
        = \left\{ f\in L_{\max\{1,p,v\}}^\loc(\R) \sep \norm{f \sep \mathcal{E}^{s}_{u,p,q}(\R)}^{(\clubsuit)}_\osc<\infty \right\},
    \end{align*}
    where
    \begin{align*}
        \norm{f \sep \mathcal{E}^{s}_{u,p,q}(\R)}^{(\clubsuit)}_\osc 
        & := \Bigg\Vert \Bigg[\Big(\int_{B(\,\cdot\,,1)} \abs{f(y)}^v \d y \Big)^{\frac{q}{v}} \\
        &\qquad  + \sum_{j=1}^{\infty} 2^{jq(s + \frac{d}{v})} \!\inf_{P \in \mathcal{P}_{N-1}} \!\!\Big( \int_{B(\,\cdot\,,2^{-j}) } \abs{f(y) - P(y)}^v\d y \Big)^{\frac{q}{v}} \Bigg]^{\frac{1}{q}} \;\Bigg|\; \mathcal{M}^u_p(\R) \Bigg\Vert
    \end{align*}
    (with the usual modifications if $q = \infty$ and/or $v = \infty$) provides an equivalent quasi-norm.
\end{satz}

\begin{proof}
    Note that we only need to show that a function $f \in L_{\min\{p,q\}}^\loc(\R)$ belongs to $\mathcal{E}^{s}_{u,p,q}(\R)$ if and only if $f \in L_v^\loc(\R)$ and
    $\norm{f \sep \mathcal{E}^{s}_{u,p,q}(\R)}^{(\clubsuit)}_\osc<\infty$ with $\norm{\cdot \sep \mathcal{E}^{s}_{u,p,q}(\R)}^{(\clubsuit)}_\osc$ being equivalent to $\norm{\cdot \sep \mathcal{E}^{s}_{u,p,q}(\R)}$ on $L_{\min\{p,q\}}^\loc(\R)$.
    Then we can argue exactly as in Step~2 of the proof of \autoref{thm:diff_Rd} above to deduce the claim.

    In order to see the equivalence, we can use the theory developed in \cite[Section~1.1--1.3]{HN}. It is not difficult to see that Triebel-Lizorkin-Morrey spaces $\mathcal{E}^s_{u,p,q}(\mathbb{R}^d)$ fit into the setting described in \cite{HN}. 
    A detailed proof for that can be found in \cite[Proposition~3.8]{Ho1}. 
    A similar observation was made also in \cite[Section~4.5.1]{ysy}. Now the above equivalence is a simple consequence of \cite[Theorem~1.1.14(iii)]{HN}. We refer to the proof of \cite[Proposition~3.9]{Ho1} for some more details.
    Note that a similar result can also be found in \cite[Corollary~4.14]{ysy}.
\end{proof}    

Let us reformulate the equivalent quasi-norm $\norm{ \cdot \sep \mathcal{E}^{s}_{u,p,q} (\mathbb{R}^d)}^{(\clubsuit) }_\osc $ obtained in \autoref{thm:osc_clubsuit} such that it gets a more convenient shape, i.e.\ in terms of the oscillation-based quantities introduced in \autoref{defi:norms}.
In particular, the following yields part (i) of our main \autoref{mainresult1}.

\begin{satz}\label{thm_osc_Rda=2}
    Let $d,N \in \mathbb{N}$, $s\in\mathbb{R}$, $0< p \leq u < \infty$, $0 < q,T,v \leq \infty$, and $0<R<\infty$ such that
    \begin{align*}
		d\, \max \left\{  0, \frac{1}{p} - 1, \frac{1}{q} - 1, \frac{1}{p} - \frac{1}{v}, \frac{1}{q} - \frac{1}{v} \right\} < s < N.
    \end{align*}
    Then 
    \begin{align*}
        \mathcal{E}^{s}_{u,p,q}(\R) 
        &= \left\{ f\in L_{\max\{1,p,v\}}^\loc(\R) \sep \norm{f \sep \mathcal{E}^{s}_{u,p,q}(\R)}^{(T,v,N)}_\osc <\infty \right\} \\
        &= \left\{ f\in L_{\max\{1,p,v\}}^\loc(\R) \sep \norm{f \sep \mathcal{E}^{s}_{u,p,q}(\R)}^{(R,T,v,N)}_\osc < \infty \right\}
    \end{align*}
    and the quasi-norms $\norm{\,\cdot \sep \mathcal{E}^{s}_{u,p,q}(\mathbb{R}^{d})}$, $\norm{\,\cdot \sep \mathcal{E}^{s}_{u,p,q}(\R)}^{(T,v,N)}_\osc$, and $\norm{\,\cdot \sep \mathcal{E}^{s}_{u,p,q}(\R)}^{(R,T,v,N)}_\osc$ are mutually equivalent on $L_{\max\{1,p,v\}}^\loc(\R)$. 
\end{satz}

\begin{proof}
    Throughout this proof w.l.o.g.\ we may assume that $q,v<\infty$. In addition, due to \autoref{lem:avg} it is sufficient to consider $R=1$.

    \emph{Step 1. } 
    We first show the assertion for $\norm{\,\cdot \sep \mathcal{E}^{s}_{u,p,q}(\R)}^{(1,T,v,N)}_\osc$. 
    To this end, in view of \autoref{thm:osc_clubsuit}, it suffices to prove that 
    \begin{align}\label{eq:proof_normeq}
		\norm{f \sep \mathcal{E}^{s}_{u,p,q}(\R)}^{(1,T,v,N)}_\osc 
		\sim \norm{f \sep \mathcal{E}^{s}_{u,p,q}(\R)}^{(\clubsuit)}_\osc,
		\qquad f\in L_{\max\{1,p,v\}}^\loc(\R).
    \end{align}
	
    \emph{Substep 1a. } 
    We start by proving \eqref{eq:proof_normeq} in the special case $0 < T \leq 1$. For this purpose, we note that for every $x\in\R$
    \begin{align*}
		\Big( \int_{0}^T \big[ t^{-s} \, \osc_{v}^{N-1} f(x,t) \big]^{q} \frac{\d t}{t} \Big)^{\frac{1}{q}}
		\leq \bigg( \sum_{j=0}^{\infty} \int_{2^{-j-1}}^{2^{-j}} t^{-sq} [ \osc_{v}^{N-1} f(x,t) ]^{q} \frac{\d t}{t} \bigg)^{\frac{1}{q}},
    \end{align*}
    where $\osc_{v}^{N-1} f(x,t) \lesssim \osc_{v}^{N-1} f(x,2^{-j})$ for $2^{-j-1}<t < 2^{-j}$, see, e.g.\ \cite[p.124]{ysy} or \cite[p.10]{Yab}. 
    Moreover, we observe
    $$
		\int_{2^{-j-1}}^{2^{-j}} t^{-sq-1}\d t
		= \left[ \frac{1}{-sq} t^{-sq} \right]_{t=2^{-j-1}}^{2^{-j}} 
		= \frac{1}{sq} [ (2^{-(j+1)})^{-sq} - (2^{-j})^{-sq} ]
		= 2^{jsq} \frac{2^{sq}-1}{sq} \sim 2^{jsq}.
    $$
    Therefore,
    \begin{align*}
		&\Big( \int_{0}^T \big[ t^{-s} \, \osc_{v}^{N-1} f(x,t) \big]^{q} \frac{\d t}{t} \Big)^{\frac{1}{q}} \\
		&\quad \lesssim \bigg( [\osc_{v}^{N-1} f(x,1) ]^{q} + \sum_{j=1}^{\infty} 2^{jsq} [ \osc_{v}^{N-1} f(x,2^{-j}) ]^{q} \bigg)^{\frac{1}{q}} \\
		&\quad \leq \bigg[ \Big( \int_{B(x,1)} \abs{f(y)}^{v} \d y \Big )^{\frac{q}{v}} + \sum_{j = 1}^{\infty} 2^{j(s + \frac{d}{v} )q}  \inf_{P \in \mathcal{P}_{N-1}} \Big( \int_{B(x , 2^{-j} ) } \abs{f(y) - P(y)}^{v} \d y \Big)^{\frac{q}{v}} \bigg]^{\frac{1}{q}},
    \end{align*}
	where we used $P=0$ to bound $\osc_{v}^{N-1} f(x,1)$.
	In addition, for every $x\in\R$,
	\begin{align*}
		&\Big( \int_{B(x,1)} \abs{f(y)}^{v} \d y \Big )^{\frac{1}{v}}\\
		&\quad \leq \bigg[ \Big( \int_{B(x,1)} \abs{f(y)}^{v} \d y \Big )^{\frac{q}{v}} + \sum_{j = 1}^{\infty} 2^{j(s + \frac{d}{v})q}  \inf_{P \in \mathcal{P}_{N-1}} \Big( \int_{B(x , 2^{-j} ) } \abs{f(y) - P(y)}^{v} \d y \Big)^{\frac{q}{v}} \bigg]^{\frac{1}{q}}.
    \end{align*}
    Together this proves ``$\lesssim$'' in \eqref{eq:proof_normeq}.

    In order to show the converse estimate, again let $x\in\R$ and choose $J\in \N$ with $2^{-J}\leq T$. 
    Then, on the one hand, we observe
    \begin{align*}
		\sum_{j=1}^{J} 2^{j(s + \frac{d}{v} )q}  \inf_{P \in \mathcal{P}_{N-1}} \Big( \int_{B(x , 2^{-j} ) } \abs{f(y) - P(y)}^{v} \d y \Big)^{ \frac{q}{v} } 
        &\leq \sum_{j=1}^{J} 2^{j(s + \frac{d}{v}  )q} \Big( \int_{B(x , 2^{-j} ) } \abs{f(y)}^{v} \d y \Big)^{\frac{q}{v}} \\
		&\lesssim \Big( \int_{B(x, 1) } \abs{f(y)}^{v} \d y \Big)^{\frac{q}{v}} .
    \end{align*}
    On the other hand, we find
    \begin{align*}
		&\sum_{j=J+1}^{\infty}\! 2^{j(s + \frac{d}{v} )q}  \!\inf_{P \in \mathcal{P}_{N-1}} \!\!\Big( \int_{B(x , 2^{-j} ) } \abs{f(y) - P(y)}^{v} \d y \Big)^{ \frac{q}{v} } \\
        &\qquad= 2^{sq} \!\sum_{j=J+1}^{\infty}\! 2^{(j-1)sq}  [ \osc_{v}^{N-1} f(x,2^{-j}) ]^{q} \\
		&\qquad\sim \sum_{j=J+1}^{\infty} \int_{2^{-j}}^{2^{-j+1}} t^{-sq-1}\d t\,  [ \osc_{v}^{N-1} f(x,2^{-j}) ]^{q} \\
		&\qquad\lesssim \sum_{j=J+1}^{\infty} \int_{2^{-j}}^{2^{-(j-1)}}  [ t^{-s}\, \osc_{v}^{N-1} f(x,t) ]^{q} \frac{\d t}{t} \\
	&\qquad\leq \int_{0}^{T}  [ t^{-s}\, \osc_{v}^{N-1} f(x,t) ]^{q} \frac{\d t}{t}.
    \end{align*}
    Consequently, a combination of both estimates yields
    \begin{align*}
		&\bigg[ \Big( \int_{B(x,1)} \abs{f(y)}^{v} \d y \Big )^{ \frac{q}{v} } + \sum_{j = 1}^{\infty} 2^{j(s + \frac{d}{v}  )q}  \inf_{P \in \mathcal{P}_{N-1}} \Big( \int_{B(x , 2^{-j} ) } \abs{f(y) - P(y)}^{v} \d y \Big)^{ \frac{q}{v}  } \bigg]^{ \frac{1}{q} } \\
		&\quad \lesssim \Big( \int_{B(x,1)} \abs{f(y)}^{v} \d y \Big )^{ \frac{1}{v} } + \Big( \int_0^T [ t^{-s}\, \osc_{v}^{N-1} f(x,t) ]^{q} \frac{\d t}{t} \Big)^{ \frac{1}{q} } .
    \end{align*}
    So we get ``$\gtrsim$'' in \eqref{eq:proof_normeq} such that the proof of \eqref{eq:proof_normeq} for $T\leq 1$ is complete.
	
    \emph{Substep 1b. }
    Due to the monotonicity of $\norm{\,\cdot \sep \mathcal{E}^{s}_{u,p,q}(\R)}^{(1,T,v,N)}_\osc$ in $T$ it remains to show
    \begin{align}\label{eq:proof_normeq1b}
		\norm{f \sep \mathcal{E}^{s}_{u,p,q}(\R)}^{(1,\infty,v,N)}_\osc 
		\lesssim \norm{f \sep \mathcal{E}^{s}_{u,p,q}(\R)}^{(1,1,v,N)}_\osc,
		\qquad f\in L_{\max\{1,p,v\}}^\loc(\R),
    \end{align}
    in order to complete Step 1. 
    For $x\in \R$ we note that similar to Step 1
    \begin{align*}
		\int_{1}^{\infty} \big[ t^{-s} \, \osc_{v}^{N-1} f(x,t) \big]^{q} \frac{\d t}{t} 
		&= \sum_{j=1}^{\infty} \int_{2^{j-1}}^{2^{j}} t^{-sq} \, [ \osc_{v}^{N-1} f(x,t) ]^{q} \frac{\d t}{t} \\
		&\lesssim \sum_{j=1}^{\infty} 2^{-jsq} \, [ \osc_{v}^{N-1} f(x,2^j) ]^{q}.
    \end{align*}
    Further note that independent of $x$ there exist appropriate displacement vectors $w_k\in\Z$, $k\in\N$, such that with  $K_j:=(2^{j+1}+1)^d\sim 2^{jd}$ there holds
    $$
	    B(x,2^j) \subset \bigcup_{k=1}^{K_j} B(x+w_k,1), \qquad j\in\N.
    $$
    Hence, choosing $P=0$, we have
    $$
	    \osc_{v}^{N-1} f(x,2^j) 
	    \leq \left( 2^{-jd} \int_{B(x,2^j)} \abs{f(y)}^v \d y \right)^{\frac{1}{v}}
	    \leq \left( \sum_{k=1}^{K_j} 2^{-jd} \int_{B(x+w_k,1)} \abs{f(y)}^v \d y \right)^{\frac{1}{v}}.
    $$
    When we set $\mu:=\min\{p,q,v\}$, using $\mu/q \leq 1$ and $\mu/v \leq 1$ we find 
    \begin{align*}
	    \Big(\int_{1}^{\infty} \big[ t^{-s} \, \osc_{v}^{N-1} f(x,t) \big]^{q} \frac{\d t}{t} \Big)^{ \frac{\mu}{q}}
	    &\lesssim \Bigg(\sum_{j=1}^{\infty} \bigg[ \sum_{k=1}^{K_j} 2^{-jsv-jd} \int_{B(x+w_k,1)} \abs{f(y)}^v \d y \bigg]^{\frac{q}{v}} \Bigg)^{\frac{\mu}{q}} \\
	    &\leq \sum_{j=1}^{\infty} \bigg[ \sum_{k=1}^{K_j} 2^{-jsv-jd} \int_{B(x+w_k,1)} \abs{f(y)}^v \d y \bigg]^{\frac{\mu}{v}} \\
	    &\leq \sum_{j=1}^{\infty}  \sum_{k=1}^{K_j} 2^{-js\mu-jd \frac{\mu}{v}  } \bigg[ \int_{B(x+w_k,1)} \abs{f(y)}^v \d y \bigg]^{ \frac{\mu}{v} }.
    \end{align*}
    Then we observe
    \begin{align*}
	\abs{f}^{(\infty,v,N)}_\osc 
        &\lesssim \norm{\Big( \int_{0}^{1} \big[ t^{-s} \, \osc_{v}^{N-1} f(\cdot,t) \big]^{q} \frac{\d t}{t} \Big)^{\frac{1}{q}} \sep \mathcal{M}^{u}_{p}(\R)} \\
	    &\qquad + \norm{\Big( \int_{1}^{\infty} \big[ t^{-s} \, \osc_{v}^{N-1} f(\cdot,t) \big]^{q} \frac{\d t}{t} \Big)^{\frac{1}{q}} \sep \mathcal{M}^{u}_{p}(\R)} \\
        &\lesssim \abs{f}^{(1,v,N)}_\osc + M,
    \end{align*}
    where 
    \begin{align*}
	    M 
	    &:= \norm{\Bigg( \sum_{j=1}^{\infty}  \sum_{k=1}^{K_j} 2^{-js\mu-jd \frac{\mu}{v} } \bigg[ \int_{B(\,\cdot\,+w_k,1)} \abs{f(y)}^v \d y \bigg]^{ \frac{\mu}{v} } \Bigg)^{ \frac{1}{\mu}  } \sep \mathcal{M}^{u}_{p}(\R)} \\
	    &= \norm{\sum_{j=1}^{\infty}  \sum_{k=1}^{K_j} 2^{-js\mu-jd \frac{\mu}{v}  } \bigg[ \int_{B(\,\cdot\,+w_k,1)} \abs{f(y)}^v \d y \bigg]^{ \frac{\mu}{v}  } \sep \mathcal{M}^{u/\mu}_{p/\mu}(\R)}^{ \frac{1}{\mu}  } \\
	    &\leq \Bigg( \sum_{j=1}^{\infty}  \sum_{k=1}^{K_j} 2^{-js\mu-jd \frac{\mu}{v}   } \norm{ \bigg( \int_{B(\,\cdot\,+w_k,1)} \abs{f(y)}^v \d y \bigg)^{ \frac{\mu}{v} } \sep \mathcal{M}^{\frac{u}{\mu}}_{\frac{p}{\mu}}(\R)} \Bigg)^{ \frac{1}{\mu} },
    \end{align*}
    see \autoref{lem:tools_M}(vi) and note that $p/\mu\geq 1$ implies that $\mathcal{M}^{u/\mu}_{p/\mu}(\R)$ is a Banach space. As it is shift invariant, we have that
    \begin{align*}
	    \norm{ \bigg( \int_{B(\,\cdot\,+w_k,1)} \abs{f(y)}^v \d y \bigg)^{ \frac{\mu}{v} } \sep \mathcal{M}^{\frac{u}{\mu}}_{\frac{p}{\mu}}(\R)} 
        &= \norm{ \bigg( \int_{B(\,\cdot\,,1)} \abs{f(y)}^v \d y \bigg)^{ \frac{\mu}{v}  } \sep \mathcal{M}^{\frac{u}{\mu}}_{\frac{p}{\mu}}(\R)} \\
        &= \norm{ \bigg( \int_{B(\,\cdot\,,1)} \abs{f(y)}^v \d y \bigg)^{ \frac{1}{v}  } \sep \mathcal{M}^u_p(\R)}^\mu 
    \end{align*}
    and hence
    \begin{align*}
	    M
        &\leq \norm{ \bigg( \int_{B(\,\cdot\,,1)} \abs{f(y)}^v \d y \bigg)^{ \frac{1}{v}  } \sep \mathcal{M}^u_p(\R)} \Bigg( \sum_{j=1}^{\infty} \sum_{k=1}^{K_j} 2^{-js\mu-jd  \frac{\mu}{v} }  \Bigg)^{ \frac{1}{\mu}  }.
    \end{align*}
    Finally, we note that due to $K_j\sim 2^{jd}$
    \begin{align*}
	    \sum_{j=1}^{\infty}  \sum_{k=1}^{K_j} 2^{-js\mu-jd \frac{\mu}{v}  } 
	    &\sim \sum_{j=1}^{\infty} 2^{jd-js\mu-jd \frac{\mu}{v} } 
	    = \sum_{j=1}^{\infty} 2^{-j\mu (s - d[ \frac{1}{\mu}   - \frac{1}{v}])} 
    \end{align*}
    converges since we assumed that
    $$
	    d \left[ \frac{1}{\mu} - \frac{1}{v} \right] 
        = d \,\max\!\left\{ 0, \frac{1}{p} - \frac{1}{v}, \frac{1}{q} - \frac{1}{v} \right\} < s.
    $$
    Together this shows $\abs{f}^{(\infty,v,N)}_\osc \lesssim \abs{f}^{(1,v,N)}_\osc + M \lesssim \norm{f \sep \mathcal{E}^{s}_{u,p,q}(\R)}^{(1,1,v,N)}_\osc$ and hence \eqref{eq:proof_normeq1b}.
    So Step 1 of the proof is complete.
	
    \emph{Step 2. } 
    We use the findings from the previous step to show that also
    \begin{align*}
		\norm{f \sep \mathcal{E}^{s}_{u,p,q}(\mathbb{R}^{d})} 
		\sim \norm{f \sep \mathcal{E}^{s}_{u,p,q}(\R)}^{(T,v,N)}_\osc,
		\qquad f\in L_{\max\{1,p,v\}}^\loc(\R).
    \end{align*}
	
    \emph{Substep 2a (lower bound). } 
    Step 1 ensures that
    $$
	    \norm{f \sep \mathcal{E}^{s}_{u,p,q}(\mathbb{R}^{d})} 
		\sim \norm{f \sep \mathcal{E}^{s}_{u,p,q}(\R)}^{(1,T,v,N)}_\osc
		\geq \abs{f}^{(T,v,N)}_\osc.
    $$
    Moreover, we can apply \autoref{thm:diff_Rd} which uses differences of higher order to see that
    $$
        \norm{f \sep \mathcal{E}^{s}_{u,p,q}(\mathbb{R}^{d})} 
		\sim \norm{f \sep \mathcal{E}^{s}_{u,p,q}(\R)}_\Delta^{(T,v,N)}
        \geq \norm{f \sep \mathcal{M}^{u}_{p}(\R)}.
    $$
	
    \emph{Substep 2b (upper bound). } 
    To complete the proof, we  distinguish two cases. If $p<v$, we have
    $$
	    d\, \max\! \left\{ 0, \frac{1}{p} - 1, \frac{1}{q} - 1, \frac{1}{q} - \frac{1}{p} \Big\} \leq d\, \max\! \Big\{ 0, \frac{1}{p} - 1, \frac{1}{q} - 1, \frac{1}{p} - \frac{1}{v}, \frac{1}{q} - \frac{1}{v} \right\}.
    $$
    In other words, we can apply Step 1 (with $v:=p$) to obtain
    \begin{align*}
	    \norm{f \sep \mathcal{E}^{s}_{u,p,q}(\mathbb{R}^{d})}
	    &\lesssim \norm{f \sep \mathcal{E}^{s}_{u,p,q}(\R)}^{(1,T,p,N)}_\osc \\
	    &=\norm{\Big( \int_{B(\,\cdot\,,1)} \abs{f(y)}^p \d y \Big )^{\frac{1}{p}} \sep \mathcal{M}^{u}_{p}(\R)} + \abs{f}^{(T,p,N)}_\osc.
    \end{align*}
	Then H\"older's inequality yields $\osc_p^{N-1} f(\cdot,t)\lesssim \osc_v^{N-1} f(\cdot,t)$ on $\R$ such that the properties of Morrey spaces allow to upper bound the second summand (up to constants) by~$\abs{f}^{(T,v,N)}_\osc$.
    In order to estimate the first summand, we use \autoref{lem:tools_M}(vii) to see that
    \begin{align*}
        \norm{\Big( \int_{B(\,\cdot\,,1)} \abs{f(y)}^p \d y \Big )^{\frac{1}{p}} \sep \mathcal{M}^{u}_{p}(\R)}
        &\lesssim \norm{f \sep \mathcal{M}^{u}_{p}(\R)}.
    \end{align*}
    
    If otherwise $v\leq p$, Step 1 yields
    \begin{align*}
	    \norm{f \sep \mathcal{E}^{s}_{u,p,q}(\mathbb{R}^{d})}
	    &\lesssim \norm{\Big( \int_{B(\,\cdot\,,1)} \abs{f(y)}^v \d y \Big )^{\frac{1}{v}} \sep \mathcal{M}^{u}_{p}(\R)} + \abs{f}^{(T,v,N)}_\osc,
    \end{align*}
    where according to \autoref{lem:tools_M}(vii) there holds
    \begin{align*}
	    \norm{\Big( \int_{B(\,\cdot\,,1)} \abs{f(y)}^v \d y \Big )^{\frac{1}{v}} \sep \mathcal{M}^{u}_{p}(\R)}
	    &\lesssim \norm{f \sep \mathcal{M}^{u}_{p}(\R)}.
    \end{align*}
    Hence, the proof is complete.
\end{proof}

\section{Characterizations of \texorpdfstring{$\mathcal{E}^{s}_{u,p,q}(\Omega)$}{Esupq(Omega)}}\label{sect:characterizations_domains}
In this section, we shall prove the main results of this paper, the intrinsic characterizations of Triebel-Lizorkin-Morrey spaces $\mathcal{E}^{s}_{u,p,q}(\Omega)$ on Lipschitz domains stated in \autoref{mainresult1}(ii) and \autoref{thm_main_2}(ii), respectively.

Let us start by proving the lower bounds. From Theorems \ref{thm:diff_Rd} and \ref{thm_osc_Rda=2} the following statement (valid for general domains) can be derived easily.
\begin{prop}[Lower bounds on domains]\label{prop:Omega_lower}
    For $d\in\N$ let $\Omega\subsetneq\R$ be any domain.
    Further, let $0< p \leq u < \infty$, $0 < q,T,v \leq \infty$, $0<R<\infty$, $N \in \mathbb{N}$, and $s\in\mathbb{R}$ be such that
    \begin{align*}
		d\, \max\! \left\{ 0, \frac{1}{p} - 1, \frac{1}{q} - 1, \frac{1}{p} - \frac{1}{v}, \frac{1}{q} - \frac{1}{v} \right\} < s < N.
    \end{align*}
    Then for $f\in \mathcal{E}^{s}_{u,p,q}(\Omega)$ there holds $f \in L_{\max\{1,p,v\}}^\loc(\Omega)$ as well as
    $$
	    \norm{f\sep \mathcal{E}^{s}_{u,p,q}(\Omega)} 
	    \gtrsim \norm{\Big( \int_{B(\,\cdot\,,R)\cap \Omega} \abs{f(y)}^{v} \d y \Big)^{\frac{1}{v}} \sep \mathcal{M}^{u}_{p}(\Omega)} 
        + \norm{ f \sep \mathcal{M}^{u}_{p}(\Omega)}
        + \abs{f}^{(T,v,N)}_{\osc,\Omega}
    $$
    and
    $$
	    \norm{f\sep \mathcal{E}^{s}_{u,p,q}(\Omega)} 
	    \gtrsim \norm{\Big( \int_{B(\,\cdot\,,R)\cap \Omega} \abs{f(y)}^{v} \d y \Big)^{\frac{1}{v}} \sep \mathcal{M}^{u}_{p}(\Omega)} 
        + \norm{ f \sep \mathcal{M}^{u}_{p}(\Omega)}
        + \abs{f}^{(T,v,N)}_{\Delta,\Omega}
    $$
    with implied constants independent of $f$.
\end{prop}
\begin{proof}
    It suffices to prove the claim for $T\geq 1$.
    To this end, let $f\in \mathcal{E}^{s}_{u,p,q}(\Omega)$. 
    Then by \autoref{def_tlmD} there exists $F\in \mathcal{E}^{s}_{u,p,q}(\R)$ such that $F\vert_\Omega = f$ on $ \Omega $ in $\mathcal{D}'(\Omega)$ and
    $$
        \norm{f\sep \mathcal{E}^{s}_{u,p,q}(\Omega)} \geq \frac{1}{2} \norm{F\sep \mathcal{E}^{s}_{u,p,q}(\R)}.
    $$
    Using \autoref{thm_osc_Rda=2} we can conclude that this extension satisfies $F\in L_{\max\{1,p,v\}}^\loc(\R)$ and $\norm{F\sep \mathcal{E}^{s}_{u,p,q}(\R)} \sim \norm{F\sep \mathcal{E}^{s}_{u,p,q}(\R)}^{(R,T,v,N)}_\osc$. 
    Therefore, $F|_\Omega \in L_{\max\{1,p,v\}}^\loc(\Omega)$ equals $f$ pointwise a.e.\ in $\Omega$ and 
    \begin{align*}
        \norm{f\sep \mathcal{E}^{s}_{u,p,q}(\Omega)} 
        &\gtrsim \norm{F\sep \mathcal{E}^{s}_{u,p,q}(\R)}^{(R,T,v,N)}_\osc \\
        &\geq \norm{\Big( \int_{B(\,\cdot\,,R)\cap \Omega} \big| F|_\Omega(y) \big|^{v} \d y \Big)^{\frac{1}{v}} \sep \mathcal{M}^{u}_{p}(\R)}\\ 
	    &\qquad 
	    + \norm{\Big( \int_{0}^{T} \big[ t^{-s} \, \osc_{v,\Omega}^{N-1} [F|_\Omega](\cdot,t) \big]^{q} \frac{\d t}{t} \Big)^{\frac{1}{q}} \sep \mathcal{M}^{u}_{p}(\R)} \\
	    &\gtrsim \norm{\Big( \int_{B(\,\cdot\,,R)\cap \Omega} \big| F|_\Omega(y) \big|^{v} \d y \Big)^{\frac{1}{v}} \sep \mathcal{M}^{u}_{p}(\Omega)} + \big|F|_\Omega\big|^{(T,v,N)}_{\osc,\Omega} \\
        &= \norm{f\sep \mathcal{E}^{s}_{u,p,q}(\Omega)}^{(R,T,v,N)}_\osc \\
        &= \norm{\Big( \int_{B(\,\cdot\,,R)\cap \Omega} \abs{f(y)}^{v} \d y \Big)^{\frac{1}{v}} \sep \mathcal{M}^{u}_{p}(\Omega)} 
        + \abs{f}^{(T,v,N)}_{\osc,\Omega},
    \end{align*}
    whereby we have used \autoref{lem:tools_M}(ii).
    In a similar fashion we obtain the lower bound $\norm{f\sep \mathcal{E}^{s}_{u,p,q}(\Omega)} \gtrsim \norm{f\sep \mathcal{E}^{s}_{u,p,q}(\Omega)}^{(T,v,N)}_\osc \gtrsim \norm{f\sep \mathcal{M}^u_p(\Omega)}$.

    Moreover, \autoref{thm:diff_Rd} shows that $\norm{f\sep \mathcal{E}^{s}_{u,p,q}(\Omega)} \gtrsim \norm{F \sep \mathcal{E}^{s}_{u,p,q}(\R)}$ is lower bounded by
    \begin{align*}
        \abs{F}^{(T,v,N)}_{\osc}\!
        &\sim\!\! \sup_{y \in \R, r > 0}\! r^{d(\frac{1}{u}-\frac{1}{p})} \bigg[ \int_{B(y,r)} \Big (  \int_{0}^{T}  t^{-sq} \Big[ t^{-d} \int_{B(0,t)} \abs{\Delta^{N}_{h}F(x)}^{v} \d h \Big]^{\frac{q}{v}} \frac{\d t}{t} \Big)^{\frac{p}{q}} \d x \bigg]^{\frac{1}{p}}  \\
        &\geq\! \sup_{y \in \Omega, r > 0}\! r^{d(\frac{1}{u}-\frac{1}{p})} \bigg[ \int_{B(y,r) \cap \Omega} \Big (\int_{0}^{T} t^{-sq} \Big[ t^{-d} \int_{V^{N}(x,t)}\!\! \abs{\Delta^{N}_{h}f(x)}^{v} \d h \Big]^{\frac{q}{v}}  \frac{\d t}{t}  \Big )^{\frac{p}{q}} \d x \bigg]^{\frac{1}{p}} \\
        &= \abs{f}^{(T,v,N)}_\Delta,
    \end{align*}
    where we used \autoref{def:morrey_mod}.
    Hence, the proof is complete.
\end{proof}

\subsection{Differences on Special Lipschitz Domains}\label{Sec_Diff1}
Let us now turn to the corresponding upper estimates. 
For that purpose, we shall start with differences on special Lipschitz domains.
The resulting characterization is not only of interest on its own sake. It will also provide a technical tool which is used to derive the assertions concerning oscillations later on.
The main ingredients of the proof will be the so-called distinguished kernels constructed by Triebel in \cite{Tr92}. 

Together with \autoref{prop:Omega_lower} above, the subsequent result especially proves part (ii) of our \autoref{thm_main_2} above.
\begin{prop}\label{prop:diff_special}
    For $d\in\N$ let $ \Omega \subset \R$ be a special Lipschitz domain. Let $0< p \leq u < \infty$, $0 < q,T \leq \infty$, $1 \leq v \leq \infty$, $0<R<\infty$, $ N \in \N$, and $s>0$.
    Then for $f \in L^\loc_{\max\{p,v\}}(\Omega)$ there holds
    \begin{align*}
        \norm{ f \sep \mathcal{E}^{s}_{u,p,q}(\Omega) } 
        & \lesssim  \norm{ \bigg(\int_{B(\,\cdot\,, R) \cap \Omega}   \abs{f(y)}^v \d y\bigg)^{\frac{1}{v}} \sep  \mathcal{M}^{u}_{p}( \Omega)} + \abs{f}^{(T,v,N)}_{\Delta,\Omega}. 
    \end{align*}
    If additionally $p\geq 1$, then also $\norm{ f \sep \mathcal{E}^{s}_{u,p,q}(\Omega) } \lesssim  \norm{ f \sep  \mathcal{M}^{u}_{p}( \Omega)} + \abs{f}^{(T,v,N)}_{\Delta,\Omega}$.
    In both cases, the implied constants are independent of $f$.
\end{prop}

\begin{proof}
    \emph{Step 1. }  In this first step we shall use intrinsic Littlewood-Paley type characterizations for the spaces $\mathcal{E}^{s}_{u,p,q}(\Omega)$. 
    For that purpose, let us recall the definition of a Littlewood-Paley family associated to a special Lipschitz domain $\Omega$. 
    We follow \cite{Yao}. 
    A sequence $\phi = ( \phi_{j} )_{j = 0}^{\infty} \subset \mathcal{S}(\mathbb{R}^d)$ of Schwartz functions is called a Littlewood-Paley family associated with $\Omega$ if the following conditions are fulfilled:
    \begin{enumerate}
        \item For all multi-indices $\alpha \in \mathbb{N}_{0}^{d}$, we have 
        \begin{align*}
            \int_{\mathbb{R}^d} x^{\alpha} \,\phi_{1}(x) \d x = 0.
        \end{align*}

        \item For $j > 1$, we have $\phi_{j} = 2^{(j-1)d}\, \phi_{1}(2^{j-1} \,\cdot\,)$.
    
        \item There holds 
        \begin{align*}
            \sum_{j = 0}^{\infty} \phi_{j} = \delta 
        \end{align*}
         with convergence in $\mathcal{S}'(\mathbb{R}^d)$, whereby $\delta$ denotes the Dirac delta distribution.

        \item For all $j \in \mathbb{N}_{0}$, we have
        \begin{align*}
            \supp (\phi_{j}) \subset \left\{ x = (x',x_{d}) \in \mathbb{R}^d \sep x_{d} < - \norm{ \abs{\nabla \omega} \sep L_{\infty}(\mathbb{R}^{d-1}) } \cdot \abs{x'} \right\}=:-K.
        \end{align*}
        Notice that $-K$ can be interpreted as reflected narrow vertically directed cone, see also \cite[Section~2]{Ry99} and \cite[Section~4]{ZHS}. 
    \end{enumerate}
    So let $\phi = ( \phi_{j} )_{j = 0}^{\infty} \subset \mathcal{S}(\mathbb{R}^d)$ be a Littlewood-Paley family associated with the special Lipschitz domain $\Omega \subset \mathbb{R}^d$. 
    Further let $0 < p \leq u < \infty$, $0 < q \leq \infty$, and $s \in \mathbb{R}$. 
    Then there exists a constant $C_{1} >0$ independent of $f \in \mathcal{S}'(\Omega)$ such that
    \begin{align}\label{eq_LPTC1_new}
        \norm{ f \sep \mathcal{E}^{s}_{u,p,q}(\Omega)} 
        \leq  C_{1} \norm{ \bigg ( \sum_{j = 0}^{\infty} 2^{jsq} \abs{(\phi_{j} \ast f)(\cdot)}^{q} \bigg)^{\frac{1}{q}} \sep \mathcal{M}^{u}_{p}(\Omega)},
    \end{align}
    where as usual $\ast$ denotes the convolution.
    This is one part of the Littlewood-Paley characterization from \cite[Theorem~1]{Yao} which generalizes a corresponding result for the original Triebel-Lizorkin spaces \cite[Theorem~3.2]{Ry99}.

    \emph{Step 2. } To continue the proof, we choose a particular Littlewood-Paley family. 
    For that purpose, we follow the ideas of Triebel; see Step~3 in the proof of \cite[Theorem~1.118]{Tr06}. 
    That is, we work with the distinguished kernels from \cite[Section~3.3.2]{Tr92} and estimate them from above as described in \cite[Formula (1.392)]{Tr06}. 
    In principle, this estimate is exactly what we want to use. However, in \cite{Tr06} the details of the proof are not given. 
    Therefore, in what follows we briefly recall the main ideas of this approach. 
    Let $N \in \mathbb{N}$. 
    Let $\varphi \in C^{\infty}_{0}(\mathbb{R})$ and $\psi \in C^{\infty}_{0}(\mathbb{R})$ be such that 
    \begin{align*}
        \int_{\mathbb{R}} \varphi(x) \d x = 1 
        \qquad \mbox{and} \qquad 
        \varphi(x) - \frac{1}{2}\, \varphi\!\left(\frac{x}{2}\right) = \psi^{(N)}(x).
    \end{align*}
    The existence of such functions was proven in \cite[Section~3.3.1/Lemma]{Tr92}.
    Now we define a function $\Phi \in C^{\infty}_{0}(\mathbb{R}^d)$ via
    \begin{align*}
        \Phi(x) := \prod_{j=1}^{d} \varphi(x_{j}), \qquad x=(x_{1} , \ldots , x_{d})  \in \mathbb{R}^d.
    \end{align*}
    Moreover, for $x\in\R$, we put
    \begin{align*}
        k_{0}(x) := \frac{(-1)^{N+1}}{N!} \sum_{r=1}^{N} \sum_{m=1}^{N} (-1)^{r+m} \binom{N}{r} \binom{N}{m}\,  \frac{m^{N-d}}{r^d}\, \Phi\!\left(\frac{x}{rm}\right) 
    \end{align*}
    and let
    \begin{align}\label{eq_scal1_new}
        k(x) := k_{0}(x) - 2^{-d}\, k_{0}\!\left(\frac{x}{2}\right),
    \end{align}
    as well as
    \begin{align}\label{eq_scal2_new}
        k_{j}(x) := 2^{jd}\, k(2^{j}x),\qquad j\in\N.
    \end{align}
     Note that the functions $\varphi$ and $\psi$ can be chosen in such a way that
    \begin{align}\label{eq_rotation_nec}
        \supp (k_{j}) \subset B(0, 2^{-j} N) \cap -K.
    \end{align}
    This follows from the construction of the involved functions (see above and \cite[Sections~3.3.1 and 3.3.2]{Tr92}) in combination with  the narrow vertically directed cone property of special Lipschitz domains; see also the explanations in \cite{Ry99}. However, to ensure \eqref{eq_rotation_nec} in the case of a small opening angle of $K$ it  becomes necessary to rotate either $\Phi$ (and hence~$k_{j}$) or to rotate $\Omega$ (and so $K$). Both strategies do not require substantial modifications in comparison with the case that no rotation is necessary. So, if we rotate $\Phi \in C^{\infty}_{0}(\mathbb{R}^d)$, the resulting function still belongs to $C^{\infty}_{0}(\mathbb{R}^d)$. Otherwise a rotation of $\Omega$ can be incorporated using \autoref{lem:tools_M}(iv) and \autoref{lem:tools_E}(ii). Therefore, in what follows it is enough to deal with the case that no rotation is required at all.
    
    Now, for $j \in \mathbb{N}_{0}$ and $f \in L^\loc_{\max\{1,p,v\}}(\Omega)$, we put
    \begin{align*}
        f_{j}(x) 
        := (k_{j} \ast f)(x)
        = \int_{B(0, 2^{-j} N) \cap -K} k_{j}(y) f(x-y) \d y 
        = \int_{B(x,2^{-j} N) \cap x + K} k_{j}(x-z) f(z) \d z .
    \end{align*}
    It is well-known that for each $x \in \Omega$ also the shifts $x + K$ are in $\Omega  $, see \cite{Ry99}. 
    With other words, $f_j$ is well-defined as we only need function values of $f$ from the inside of $\Omega$; see also (the beginning of) \cite[Section~2]{Ry99}. 
    In \cite[Formula (10) on p.175]{Tr92} it was observed that for $j \in \mathbb{N}$ we can write
    \begin{align}\label{eq_kern_diff1_new}
        f_{j}(x) = \sum_{| \alpha | = N} \int_{\mathbb{R}^d} D^{\alpha} k_{\alpha}(-y) \, \Delta^{N}_{2^{-j}y} f(x) \d y,
    \end{align}
    with appropriate $k_{\alpha} \in C^{\infty}_{0}(\mathbb{R}^d)$. Notice that in \eqref{eq_kern_diff1_new} there is an additional ``$-$'' since we used a slightly different definition for $f_{j}$. 
    Moreover, recall that $( k_{j} )_{j = 0}^{\infty}$ can be interpreted as a Littlewood-Paley family. 
    This already has been observed in Step~3 of the proof of \cite[Theorem~1.118]{Tr06}. (The moment condition can be derived from \cite[Formula (8) on p.174]{Tr92} in combination with \cite[Remark after (1.1)]{Ry99}. 
    The scaling condition directly follows from our definition \eqref{eq_scal2_new}. 
    The approximation identity can be derived from the definitions of $k_{j}$ and $k$; see also \cite[Formula (9) on p.174]{Tr92}. 
    Finally, the support condition can be fulfilled by an appropriate choice of the functions $\varphi$ and $\psi$, see above. 
    Here we also refer to the explanations concerning \cite[Figure~1]{Ry99}.) 
    
    Summing up all we did up to now in Steps~1 and~2, Formula~\eqref{eq_LPTC1_new} becomes
    \begin{equation}\label{eq_LPTC2_new}
        \norm{ f \sep \mathcal{E}^{s}_{u,p,q}(\Omega)} 
        \lesssim \norm{ \bigg( \sum_{j = 0}^{\infty} 2^{jsq} \abs{f_{j}(\cdot)}^{q} \bigg)^{\frac{1}{q}} \sep \mathcal{M}^{u}_{p}(\Omega)}.
    \end{equation}

    \emph{Step 3. } 
    Next we show that the functions $\abs{f_{j}}$ can be estimated from above by integrals of higher order differences of $f$. 
    This observation already has been made in \cite[Formula~(1.392)]{Tr06}.
    Nevertheless, we will give some details.
    
    We start with the special case $j\leq J$, where $J\in\N$ is chosen such that $2^{-J} \leq T$. 
    Note that for $x \in \Omega$ Formulas \eqref{eq_scal1_new} and \eqref{eq_scal2_new} yield
    $$
        \abs{k_j(x)} \leq 2^{jd} \abs{k_0(2^j x)} + 2^{(j-1)d} \abs{k_0(2^{j-1} x)}, \qquad j\in\N,
    $$
    such that for $j=0,\ldots,J$ we obtain $\abs{k_j(x)} \lesssim \sum_{\ell=0}^J \abs{k_0(2^{\ell} x)}$ with a constant depending on $d$ and $J$.
    Hence, for $x\in\Omega$,
    \begin{align*}
        \abs{f_j(x)}
        &\lesssim \int_{B(x, 2^{-j} N) \cap \Omega} \abs{k_{j}(x-y)} \abs{f(y)} \d y \\
        &\lesssim \int_{B(x, N) \cap \Omega} \sum_{\ell=0}^J \abs{k_{0}\!\left(\frac{x-y}{2^{-\ell}}\right)} \abs{f(y)} \d y \\
        &=\! \int_{B(x, N) \cap \Omega} \sum_{\ell=0}^J \abs{ \frac{(-1)^{N+1}}{N!} \sum_{r=1}^{N} \sum_{m=1}^{N} (-1)^{r+m} \binom{N}{r} \binom{N}{m}  \frac{m^{N-d}}{r^d} \Phi\!\left(\frac{x-y}{2^{-\ell}\, rm}\right)} \abs{f(y)} \d y \\
        &\lesssim \int_{B(x, N) \cap \Omega}  \sum_{\ell=0}^J \sum_{r=1}^{N} \sum_{m=1}^{N}  \abs{\Phi\!\left(\frac{x-y}{2^{-\ell}\, rm}\right)} \abs{f(y)} \d y \\
        &\lesssim \int_{B(x, N) \cap \Omega} \abs{f(y)} \d y.
    \end{align*}
    
    Let us now turn to the case $j > J$ for which we can use the alternative representation of $f_{j}$ given by \eqref{eq_kern_diff1_new}. Recall that also the functions $k_{\alpha}$ are defined in terms of $\Phi$, see \cite[Formula (26) in Chapter 3.3.2]{Tr92}. 
    Consequently, they can be chosen in a way such that $\supp k_{\alpha} \subset B(0,1)  \cap -K $. For $x \in \Omega$ using $k_{\alpha} \in C^{\infty}_{0}(\mathbb{R}^d)$ the narrow vertically directed cone property of $\Omega$ yields
    \begin{align*}
        \abs{f_{j}(x)} 
        &\leq \sum_{\abs{\alpha} = N} \int_{B(0,1) \cap \{ y\in\R \sep x + \ell 2^{-j}  y \in \Omega \; \text{for} \; 0 \leq \ell \leq N  \}} \abs{D^{\alpha} k_{\alpha}(-y)} \abs{\Delta^{N}_{2^{-j}y} f(x)} \d y \\
        &\lesssim \int_{B(0,1) \cap \{ y\in\R \sep x + \ell 2^{-j}  y \in \Omega \; \text{for} \; 0 \leq \ell \leq N  \}} \abs{\Delta^{N}_{2^{-j}y} f(x)} \d y \\
        &= 2^{jd} \int_{B(0,2^{-j}) \cap \{ h\in\R \sep x + \ell h \in \Omega \; \text{for} \; 0 \leq \ell \leq N \}} \abs{\Delta^{N}_{h} f(x)} \d h,
    \end{align*}
    where we put $h := 2^{-j} y$. 
    This also verifies \cite[Formula (1.392)]{Tr06}. 

    \emph{Step 4. }
    A combination of \eqref{eq_LPTC2_new} with the estimates from Steps~3 yields
    \begin{align*}
        \norm{f \sep \mathcal{E}^{s}_{u,p,q}(\Omega)} 
        &\lesssim \norm{ \int_{B(\,\cdot\,, N) \cap \Omega} \abs{f(y)} \d y \sep  \mathcal{M}^{u}_{p}( \Omega)} \\
        &\qquad + \norm{\bigg( \sum_{j=J+1}^{\infty} 2^{jsq} \Big (  2^{jd}  \int_{V^{N}(\,\cdot\,,2^{-j})} \abs{\Delta^{N}_{h} f(\cdot)} \d h \Big)^{q} \bigg)^{\frac{1}{q}} \sep \mathcal{M}^{u}_{p}(\Omega)},
    \end{align*}
    where \autoref{lem:avg} can be used to replace $B(\,\cdot\,, N)$ by $B(\,\cdot\,, R)$ and for every $x\in\Omega$ the sum can be bounded by 
    \begin{align*}
        &\sum_{j = J+1}^{\infty} 2^{j} \int_{2^{-j}}^{2^{-j+1}} \, 2^{jsq} \Big( 2^{jd}  \int_{V^{N}(x,2^{-j})} \abs{\Delta^{N}_{h} f(x)} \d h \Big)^{q} \d t\\
        &\qquad \lesssim  \sum_{j = J+1}^{\infty} \int_{2^{-j}}^{2^{-j+1}}  t^{-sq} \Big(  t^{-d} \int_{V^{N}(x,t)} \abs{\Delta^{N}_{h} f(x)} \d h \Big)^{q} \frac{\d t}{t}  \\
        &\qquad \leq \int_{0}^{T}  t^{-sq} \Big( t^{-d}  \int_{V^{N}(x,t) }  \abs{\Delta^{N}_{h} f(x)} \d h \Big)^{q} \frac{\d t}{t}
    \end{align*}
    since $2^{-J} \leq T$. 
    Hence, we obtain
    \begin{align*}
        \norm{ f \sep \mathcal{E}^{s}_{u,p,q}(\Omega) } 
        &\lesssim  \norm{ \int_{B(\cdot, R) \cap \Omega}   \abs{f(y)} \d y \sep  \mathcal{M}^{u}_{p}( \Omega)} + \abs{f}^{(T,1,N)}_{\Delta,\Omega}
    \end{align*}
    which coincides with the desired result if $v=1$.
    
    \emph{Step 5. }
    If $v>1$, we note that for every $x\in\Omega$ and $0<t<1$ H\"older's inequality gives
    $$
        \int_{B(x,R)\cap \Omega} \abs{f(y)} \d y
        \lesssim \left( \int_{B(x,R)\cap \Omega} \abs{f(y)}^v \d y \right)^{\frac{1}{v}}
    $$
    as well as
    $$
         t^{-d}  \int_{V^{N}(x,t) } \abs{ \Delta^{N}_{h} f(x)} \d h
        \lesssim \left( t^{-d}  \int_{V^{N}(x,t) } \abs{ \Delta^{N}_{h} f(x)}^{v} \d h \right)^{\frac{1}{v}}
    $$
    since $\abs{V^{N}(x,t)}\lesssim t^d$.

    If we additionally assume that $p\geq 1$, we can modify the previous argument and use \autoref{lem:tools_M}(ii), (vii), and (i) to see that
    \begin{align*}
       \norm{ \int_{B(\,\cdot\,,R)\cap \Omega} \abs{f(y)} \d y  \sep \mathcal{M}^{u}_{p}(\Omega)}
       &\leq  \norm{ \int_{B(\,\cdot\,,R)} \abs{Ef(y)} \d y  \sep \mathcal{M}^{u}_{p}(\mathbb{R}^d)} \\
       &\lesssim \norm{ Ef  \sep \mathcal{M}^{u}_{p}(\R)} \\
       &\lesssim \norm{ f  \sep \mathcal{M}^{u}_{p}(\Omega)}
    \end{align*}
    which completes the proof.
\end{proof}

Later we shall also prove characterizations in terms of higher order differences for Triebel-Lizorkin-Morrey spaces on convex bounded Lipschitz domains. 
However, to avoid technical difficulties, at first we deduce the desired characterizations by local oscillations and come back to this issue in \autoref{Sec_Diff1_re} below.

\subsection{Oscillations on Lipschitz Domains}\label{sec_osc_dom1}
Here we prove characterizations in terms of local oscillations for Triebel-Lizorkin-Morrey spaces $\mathcal{E}^{s}_{u,p,q}(\Omega)$ on special or bounded Lipschitz domains $\Omega\subset\R$ with $d\in\N$. 

In view of \autoref{prop:Omega_lower} above, it suffices to prove corresponding upper bounds.
To this end, we first concentrate on special Lipschitz domains. 
In preparation for that we collect some facts concerning projections onto quasi-optimal polynomials for which we follow the ideas developed in \cite[Section~2]{DeSh}.

\begin{lem}[Quasi-optimal polynomials]
\label{lem:opt_pol}
    Let $d,N\in\N$ and $\Omega \subset \mathbb{R}^d$ be a special Lipschitz domain. 
    For $x\in\Omega$ and $t>0$ consider the Hilbert spaces $\mathcal{H}_{x,t}^{N-1} := \big( \mathcal{P}_{N-1}, \distr{\cdot}{\cdot}_{x,t}\big)$ with inner product
    $$
        \distr{f}{g}_{x,t} := t^{-d} \int_{B(x,t)\cap \Omega} f(z)\, \overline{g(z)} \d z
    $$
    and for some appropriate index set $\mathcal{I}$ let $\{p_{i,x,t} \sep i\in \mathcal{I}\}$ be an orthonormal basis of $\mathcal{H}_{x,t}^{N-1}$ such that 
    $$
        \sup_{y\in B(x,t)\cap \Omega} \abs{p_{i,x,t}(y)} \leq C
    $$
    for some $C>0$ independent of $x$ and $t$ (see Remark \ref{rem_gram_schm} below). 
    Then there exist constants $c_1=c_1(C^2,\#\mathcal{I})>0$ and $c_2=c_2(c_1, d)>0$ such that the projection operators
    $$
        \Pi_{x,t}^{N-1} \colon L_1(B(x,t)\cap \Omega) \to \mathcal{H}_{x,t}^{N-1}, \qquad f\mapsto \Pi_{x,t}^{N-1} f := \sum_{i\in \mathcal{I}} \distr{f}{p_{i,x,t}}_{x,t} \, p_{i,x,t},
    $$
    satisfy 
    \begin{enumerate}
        \item the pointwise bound
        $$
            \abs{\big(\Pi_{x,t}^{N-1}f\big)(y)} \leq c_1 \, t^{-d} \int_{B(x,t)\cap \Omega} \abs{f(z)} \d z, \qquad y\in B(x,t)\cap \Omega,
        $$

        \item for all $p\in\mathcal{P}_{N-1}$ the equation $ \Pi_{x,t}^{N-1}[f-p] = \Pi_{x,t}^{N-1}f-p \in \mathcal{P}_{N-1}$,
        
        \item the quasi-optimality
        $$
            t^{-d} \int_{B(x,t)\cap \Omega} \abs{f(z) - \big(\Pi_{x,t}^{N-1}f\big)(z)} \d z 
            \leq c_2 \, \osc_{1,\Omega}^{N-1}f(x,t),
        $$
        
        \item the limit property $\lim_{t\to 0}\limits \big(\Pi_{x,t}^{N-1}f\big)(x) = f(x)$ for a.e.\ $x\in\Omega$.
    \end{enumerate}
\end{lem}

\begin{rem}\label{rem_gram_schm}
    An orthonormal basis $\{p_{i,x,t} \sep i\in \mathcal{I}\}$ can be constructed out of appropriate monomials using the Gram-Schmidt process, see also \cite[Section~2]{DeSh}.
\end{rem}

\begin{proof}
    Let $x \in \Omega$ and $t > 0$. Let $f\in L_1(B(x,t)\cap \Omega)$.
    \begin{enumerate}
        \item For $y\in B(x,t)\cap \Omega$ we have
        \begin{align*}
            \abs{\big(\Pi_{x,t}^{N-1}f\big)(y)} 
             \leq \sum_{i\in \mathcal{I}} t^{-d}  \int_{B(x,t)\cap\Omega} \abs{f(z)} \, C \d z \, C = c_1\, t^{-d} \int_{B(x,t)\cap\Omega} \abs{f(z)} \d z.
        \end{align*}

        \item $\Pi_{x,t}^{N-1}$ is linear and satisfies $\Pi_{x,t}^{N-1}=\mathrm{id}$ on $\mathcal{H}_{x,t}^{N-1} \ni p$.
        
        \item Since $\mathcal{H}_{x,t}^{N-1}$ is a linear subspace of the normed space $L_1(B(x,t)\cap \Omega)$, there exists a best\-approximation $\pi\in \mathcal{H}_{x,t}^{N-1}$ to $f$.
        Hence, (ii) implies
        \begin{align*}
            &t^{-d} \int_{B(x,t)\cap \Omega} \abs{f(y) - \big(\Pi_{x,t}^{N-1}f\big)(y)} \d y \\
            &\leq t^{-d} \int_{B(x,t)\cap \Omega} \abs{f(y) - \pi(y)} \d y
            + t^{-d} \int_{B(x,t)\cap \Omega} \abs{\big(\Pi_{x,t}^{N-1}f\big)(y)-\pi(y)} \d y \\
            &= \osc_{1,\Omega}^{N-1}f(x,t)
            + t^{-d} \int_{B(x,t)\cap \Omega} \abs{\big(\Pi_{x,t}^{N-1}[f-\pi]\big)(y)} \d y
        \end{align*}
        and we can use (i) to conclude
        \begin{align*}
            t^{-d} \!\int_{B(x,t)\cap \Omega} \abs{\big(\Pi_{x,t}^{N-1}[f-\pi]\big)(y)} \d y
            &\leq t^{-d} \!\int_{B(x,t)\cap \Omega} c_1\, t^{-d} \int_{B(x,t)\cap\Omega} \abs{(f-\pi)(z)} \d z \d y \\
            &= \osc_{1,\Omega}^{N-1}f(x,t) \, t^{-d} \int_{B(x,t)\cap \Omega} c_1 \d y,
        \end{align*}
        where
        $$
            t^{-d} \int_{B(x,t)\cap \Omega} c_1 \d y \leq c_1 \, t^{-d} \abs{B(x,t)} =:c_2-1.
        $$

        \item Since $\Omega$ is open and $x \in \Omega$, there exists $t_{0}>0$, such that for $t<t_0$ there holds $B(x,t)\subset\Omega$. Consequently in the limiting case we can argue as in \cite[Formula (2.7)]{DeSh}, where the assertion is proved.\qedhere
    \end{enumerate}
\end{proof}

Now we are well-prepared to derive the desired upper estimates for special Lipschitz domains which together with \autoref{prop:Omega_lower} imply the corresponding part of \autoref{mainresult1}(ii). 
For this we use some tools provided in the proof of \autoref{prop:diff_special}, i.e., we deal with higher order differences. It turns out that they are closely related to local oscillations. 

\begin{prop}\label{prop:osc_special}
    For $d\in\N$ let $ \Omega \subset \R$ be a special Lipschitz domain. Let $0< p \leq u < \infty$, $0 < q,T \leq \infty$, $1 \leq v \leq \infty$, $0<R<\infty$, $ N \in \N$, and $s>0$.
    Then for $f \in L^\loc_{\max\{p,v\}}(\Omega)$ there holds
    \begin{align*}
        \norm{ f \sep \mathcal{E}^{s}_{u,p,q}(\Omega) } 
        & \lesssim  \norm{ \bigg(\int_{B(\,\cdot\,, R) \cap \Omega}   \abs{f(y)}^v \d y\bigg)^{\frac{1}{v}} \sep  \mathcal{M}^{u}_{p}( \Omega)} + \abs{f}^{(T,v,N)}_{\osc,\Omega}. 
    \end{align*}
    If additionally $p\geq 1$, then also $\norm{ f \sep \mathcal{E}^{s}_{u,p,q}(\Omega) } \lesssim  \norm{ f \sep  \mathcal{M}^{u}_{p}( \Omega)} + \abs{f}^{(T,v,N)}_{\osc,\Omega}$.
    In both cases, the implied constants are independent of $f$.
\end{prop}

\begin{proof}
    W.l.o.g.\ we can assume that $T<N$, as well as $N>s$. 
    
    \emph{Step 1. } For $f \in L^\loc_{\max\{p,v\}}(\Omega)$
    \autoref{prop:diff_special} (with $T=v=1$) yields 
    \begin{align*}
        \norm{ f \sep \mathcal{E}^{s}_{u,p,q}(\Omega) } 
        & \lesssim  \norm{ \int_{B(\,\cdot\,, R) \cap \Omega} \abs{f(y)} \d y \sep  \mathcal{M}^{u}_{p}( \Omega)} + \abs{f}^{(1,1,N)}_{\Delta,\Omega}.
    \end{align*}
    In this first step, we are going to estimate the averaged differences in $\abs{f}^{(1,1,N)}_{\Delta,\Omega}$ from above in terms of polynomials and local oscillations of $f$. 
    For that purpose, we let $x \in \Omega$ be fixed and use some ideas from \cite[Lemma 4.10]{ysy}, see also the proof of \cite[Theorem 1]{See1}. At first, we note that a change of measure yields
    \begin{align*}
        &\int_{0}^{1}  t^{-sq} \bigg ( t^{-d}  \int_{V^{N}(x,t) } \abs{ \Delta^{N}_{h} f(x)} \d h \bigg)^{q} \frac{\d t}{t} \\
        &\qquad = \int_{0}^{N}  \left(\frac{\tau}{N}\right)^{-sq} \bigg ( \left(\frac{\tau}{N}\right)^{-d} \int_{V^{N}\!\left(x,\frac{\tau}{N}\right)} \abs{ \Delta^{N}_{h} f(x)} \d h \bigg)^{q} \frac{\d \tau}{\tau} \\
        &\qquad \sim \int_{0}^{N}  t^{-sq} \bigg ( t^{-d} \int_{V^{N}\!\left(x,\frac{t}{N}\right)} \abs{\Delta^{N}_{h} f(x)} \d h \bigg)^{q} \frac{\d t}{t}.
    \end{align*}
    
    Second, Taylor's theorem easily shows that for each $h\in\R\setminus\{0\}$ and $p\in\mathcal{P}_N$, there holds $\Delta^{1}_h p \in\mathcal{P}_{N-1}$.
    Therefore, $\Delta_h^N p \in \mathcal{P}_0$ is constant on $\R$ such that $\Delta_h^{N+1}p \equiv 0$.
    Together with \eqref{eq:Delta} this shows that for every $t>0$, all $h\in V^N(x,\frac{t}{N})$, and each $p\in\mathcal{P}_{N-1}$ we have
    \begin{align*}
        \abs{\Delta_h^Nf(x)} 
        &= \abs{\Delta_h^N[f-p](x)} \\
        &= \abs{\sum_{k=0}^N (-1)^{N-k} \binom{N}{k} \,[f-p](x+kh)}\\
        &\lesssim \abs{[f-p](x)} + \sum_{k=1}^N \abs{[f-p](x+kh)}
    \end{align*}
    and thus
    \begin{align*}
        \int_{V^{N}\!\left(x,\frac{t}{N}\right)} \abs{\Delta^{N}_{h} f(x)} \d h
        \lesssim \int_{V^{N}\!\left(x,\frac{t}{N}\right)} \abs{[f-p](x)} \d h + \sum_{k=1}^N \int_{V^{N}\!\left(x,\frac{t}{N}\right)} \abs{[f-p](x+kh)} \d h.
    \end{align*}
    Therein, we have
    $$
        \int_{V^{N}\!\left(x,\frac{t}{N}\right)} \abs{[f-p](x)} \d h 
        \leq \abs{[f-p](x)} \abs{B\!\left(0,\frac{t}{N}\right)} 
        \sim t^d\, \abs{[f-p](x)}
    $$
    as well as for $k=1,\ldots,N$
    \begin{align*}
        \int_{V^{N}\!\left(x,\frac{t}{N}\right)} \abs{[f-p](x+kh)} \d h
        &\leq \int_{B\left(0,\frac{t}{N}\right)} \chi_{\Omega}(x+kh) \abs{[f-p](x+kh)} \d h \\
        &\sim \int_{B\left(0,\frac{kt}{N}\right)} \chi_{\Omega}(x+\widetilde{h}) \abs{[f-p](x+\widetilde{h})} \d \widetilde{h} \\
        &\leq \int_{B(x,t)} \chi_{\Omega}(y) \abs{[f-p](y)} \d y \\
        &= \int_{B(x,t)\cap \Omega} \abs{[f-p](y)} \d y.
    \end{align*}
    Hence, we conclude that for all $t>0$ and each $p\in\mathcal{P}_{N-1}$
    \begin{align*}
        t^{-d} \int_{V^{N}\!\left(x,\frac{t}{N}\right)} \abs{\Delta^{N}_{h} f(x)} \d h
        \lesssim \abs{[f-p](x)} + t^{-d} \int_{B(x,t)\cap \Omega} \abs{[f-p](y)} \d y.
    \end{align*}
    In particular, we can choose $p:=\Pi_{x,t}^{N-1}f$ from \autoref{lem:opt_pol} such that the second term can be replaced by $\osc_{1,\Omega}^{N-1}f(x,t)$. In conclusion, this shows
    \begin{align*}
        \abs{f}^{(1,1,N)}_{\Delta,\Omega} 
        &\lesssim \norm{ \bigg( \int_{0}^{N} t^{-sq} \abs{\big[f-\Pi_{(\cdot),t}^{N-1}f\big](\cdot)}^{q} \frac{\d t}{t} \bigg)^{\frac{1}{q}} \sep  \mathcal{M}^{u}_{p}( \Omega)} + \abs{f}^{(N,1,N)}_{\osc,\Omega}.
    \end{align*}

    \emph{Step 2. }
    We still need to estimate $\abs{\big[f-\Pi_{(\cdot),t}^{N-1}f\big](\cdot)}$ in terms of oscillations. 
    To this end, we employ some ideas from \cite[p.112]{ysy}.  
    Again let $x\in\Omega$ and $t>0$ be fixed and note that for all $L\in\N$
    $$
        \abs{\big[f-\Pi_{x,t}^{N-1}f\big](x)} 
        \leq \abs{f(x)-\big(\Pi_{x,2^{-L}t}^{N-1}f\big)(x)} 
        + \sum_{\ell=0}^{L-1} \abs{\big(\Pi_{x,2^{-(\ell+1)}t}^{N-1}f\big)(x)-\big(\Pi_{x,2^{-\ell}t}^{N-1}f\big)(x)},
    $$
    where, due to \autoref{lem:opt_pol}(ii), (i), and (iii),
    \begin{align*}
        \abs{\big(\Pi_{x,2^{-(\ell+1)}t}^{N-1}f\big)(x)-\big(\Pi_{x,2^{-\ell}t}^{N-1}f\big)(x)}
        &= \abs{\big(\Pi_{x,2^{-(\ell+1)}t}^{N-1}\big[f-\Pi_{x,2^{-\ell}t}^{N-1}f\big]\big)(x)} \\
        &\lesssim \left( 2^{-(\ell+1)}t \right)^{-d} \int_{B(x,2^{-(\ell+1)}t)\cap \Omega} \abs{\big[f-\Pi_{x,2^{-\ell}t}^{N-1}f\big](z)} \d z \\
        &\lesssim \left( 2^{-\ell}t \right)^{-d} \int_{B(x,2^{-\ell}t)\cap \Omega} \abs{f(z)-\big(\Pi_{x,2^{-\ell}t}^{N-1}f\big)(z)} \d z \\
        &\lesssim \osc_{1,\Omega}^{N-1}f(x,2^{-\ell}t)
    \end{align*}
    with constants that do not depend on $f$, $x$, $t$, or $\ell$.
    Thus,
    $$
        \sum_{\ell=0}^{L-1} \abs{\big(\Pi_{x,2^{-(\ell+1)}t}^{N-1}f\big)(x)-\big(\Pi_{x,2^{-\ell}t}^{N-1}f\big)(x)}
        \lesssim \sum_{\ell=0}^{\infty} \osc_{1,\Omega}^{N-1}f(x,2^{-\ell}t).
    $$
    If $L\in\N$ is chosen large enough, \autoref{lem:opt_pol}(iv) ensures that also $\abs{f(x)-\big(\Pi_{x,2^{-L}t}^{N-1}f\big)(x)}$ is smaller than this quantity such that we derive that
    $$
        \abs{\big[f-\Pi_{x,t}^{N-1}f\big](x)} 
        \lesssim \sum_{\ell=0}^{\infty} \osc_{1,\Omega}^{N-1}f(x,2^{-\ell}t).
    $$
    
    Next, it is easily seen that for $m:=\min\{1,q\}$
    $$
        \bigg( \int_{0}^{N} t^{-sq} \abs{g_1(t)+g_2(t)}^{q} \frac{\d t}{t} \bigg)^{\frac{m}{q}}
        \leq \bigg( \int_{0}^{N} t^{-sq} \abs{g_1(t)}^{q} \frac{\d t}{t} \bigg)^{\frac{m}{q}} + \bigg( \int_{0}^{N} t^{-sq} \abs{g_2(t)}^{q} \frac{\d t}{t} \bigg)^{\frac{m}{q}}.
    $$
    Therefore,
    \begin{align*}
        \bigg( \int_{0}^{N} t^{-sq} \abs{\big[f-\Pi_{x,t}^{N-1}f\big](x)}^{q} \frac{\d t}{t} \bigg)^{\frac{m}{q}} 
        &\lesssim \bigg( \int_{0}^{N} t^{-sq} \Big( \sum_{\ell=0}^{\infty} \osc_{1,\Omega}^{N-1}f(x,2^{-\ell}t) \Big)^{q} \frac{\d t}{t} \bigg)^{\frac{m}{q}} \\
        &\leq \sum_{\ell=0}^{\infty} \bigg( \int_{0}^{N} t^{-sq} \, \osc_{1,\Omega}^{N-1}f(x,2^{-\ell}t)^{q} \frac{\d t}{t} \bigg)^{\frac{m}{q}} \\
        &=\sum_{\ell=0}^{\infty} \bigg( 2^{-\ell s q} \int_{0}^{2^{-\ell}N} [\tau^{-s} \, \osc_{1,\Omega}^{N-1}f(x,\tau)]^{q} \frac{\d \tau}{\tau} \bigg)^{\frac{m}{q}} \\
        &\lesssim \bigg( \int_{0}^{N} [t^{-s} \,\osc_{1,\Omega}^{N-1}f(x,t)]^{q} \frac{\d t}{t} \bigg)^{\frac{m}{q}},
    \end{align*}
    as the geometric series converges, due to $s>0$ and $m>0$. So, we have shown that
    \begin{align*}
        \abs{f}^{(1,1,N)}_{\Delta,\Omega} 
        &\lesssim \norm{ \bigg( \int_{0}^{N} t^{-sq} \abs{\big[f-\Pi_{(\cdot),t}^{N-1}f\big](\cdot)}^{q} \frac{\d t}{t} \bigg)^{\frac{1}{q}} \sep  \mathcal{M}^{u}_{p}( \Omega)} + \abs{f}^{(N,1,N)}_{\osc,\Omega}
        \lesssim \abs{f}^{(N,1,N)}_{\osc,\Omega}.
    \end{align*}
    Since we assumed $T<N<\infty$, for every fixed $x\in\Omega$, we may write
    \begin{align*}
        \int_{T}^{N} \big[t^{-s} \, \osc_{1,\Omega}^{N-1}f(x,t)\big]^{q} \frac{\d t}{t}
        &\leq \int_{T}^{N} \Big[t^{-s} \, t^{-d} \int_{B(x,t)\cap\Omega} \abs{f(y)}\d y\Big]^{q} \frac{\d t}{t} \\
        &\lesssim \left( \int_{B(x,N)\cap\Omega} \abs{f(y)}\d y \right)^q
    \end{align*}
    such that finally
    \begin{align*}
        \abs{f}^{(1,1,N)}_{\Delta,\Omega}
        &\lesssim \norm{ \bigg( \int_{0}^{T} \big[t^{-s} \, \osc_{1,\Omega}^{N-1}f(\cdot,t)\big]^{q} \frac{\d t}{t} \bigg)^{\frac{1}{q}} \sep  \mathcal{M}^{u}_{p}( \Omega)} 
        + \norm{ \int_{B(\,\cdot\,, R) \cap \Omega} \abs{f(y)} \d y \sep  \mathcal{M}^{u}_{p}( \Omega)},
    \end{align*}
    where we used \autoref{lem:avg} to replace $B(\,\cdot\,, N)$ by $B(\,\cdot\,, R)$.

    \emph{Step 3. } Combining the previous steps shows the claim with $v=1$, i.e.
    \begin{align*}
        \norm{ f \sep \mathcal{E}^{s}_{u,p,q}(\Omega) } 
        & \lesssim  \norm{ \int_{B(\,\cdot\,, R) \cap \Omega} \abs{f(y)} \d y \sep  \mathcal{M}^{u}_{p}( \Omega)}  + \abs{f}^{(T,1,N)}_{\osc,\Omega}
    \end{align*}
    and we can complete the proof similar to Step~5 in the proof of \autoref{prop:diff_special}.
\end{proof}

Finally, we transfer our previous findings to the case of \emph{bounded} Lipschitz domains in order to complete the proof of \autoref{mainresult1}(ii).
\begin{prop}\label{prop:boundedLip_osc_RTv}
    For $d\in\N$ let $ \Omega \subset \R$ be a bounded Lipschitz domain. Let $0< p \leq u < \infty$, $0 < q,T \leq \infty$, $1 \leq v \leq \infty$, $0<R<\infty$, $ N \in \N$ and $s>\sigma_{p,q}$.
    Then for $f \in L^\loc_{\max\{p,v\}}(\Omega)$ there holds
    \begin{align*}
        \norm{ f \sep \mathcal{E}^{s}_{u,p,q}(\Omega) } 
        & \lesssim  \norm{ \bigg(\int_{B(\,\cdot\,, R) \cap \Omega}   \abs{f(y)}^v \d y\bigg)^{\frac{1}{v}} \sep  \mathcal{M}^{u}_{p}( \Omega)} + \abs{f}^{(T,v,N)}_{\osc,\Omega}. 
    \end{align*}
    If additionally $p\geq 1$, then also $\norm{ f \sep \mathcal{E}^{s}_{u,p,q}(\Omega) } \lesssim  \norm{ f \sep  \mathcal{M}^{u}_{p}( \Omega)} + \abs{f}^{(T,v,N)}_{\osc,\Omega}$.
    In both cases, the implied constants are independent of $f$.
\end{prop}

\begin{proof}
    W.l.o.g.\ we may assume that $T<\infty$. Further, we can assume $q<\infty$ as otherwise the usual modifications have to be made. 

    \emph{Step 1 (Localization). } 
    Let $\Omega\subset\R$ be a bounded Lipschitz domain and assume that $0<p\leq u<\infty$, as well as $0<q\leq\infty$, and $s\in\re$.
    First, let us show by standard arguments that for $f\in \mathcal{D}'(\Omega)$ and an arbitrary collection of extensions $F_k\in\mathcal{S}'(\R)$ to $f$,
    \begin{align}\label{eq:proof_splitting}
        \norm{f \sep \mathcal{E}^{s}_{u,p,q}(\Omega)}
        \lesssim \norm{ \sigma_0 F_0 \sep \mathcal{E}^{s}_{u,p,q}(\R)} + \sum_{k=1}^m \norm{ \big[(\sigma_k F_k) \circ \Phi_k^{-1}\big]\vert_{\omega_k} \sep \mathcal{E}^{s}_{u,p,q}(\omega_k)}
    \end{align}
    with $\omega_k$ being special Lipschitz domains. The assumption on $\Omega$ implies that there exist open balls $B_1,\ldots, B_m$ in $\R$, affine-linear diffeomorphisms $\Phi_1,\ldots,\Phi_m\colon\R\to\R$, and $[0,1]$-valued functions $\sigma_1,\ldots,\sigma_m\in\mathcal{D}(\R)$ with the following properties for $k=1,\ldots,m$:
    \begin{itemize}
        \item $B_k \cap\partial\Omega\neq\emptyset$ and $\partial\Omega\subset\bigcup_{k=1}^m B_k$,
        \item $\supp \sigma_k \subset B_k$ and $\sum_{k=1}^m\sigma_k \equiv 1$ on some neighborhood of $\partial\Omega$,
        \item $\Phi_k(B_k)\cap \Phi_k(\Omega)$ can be extended to a special Lipschitz domain $\omega_k \subset\R$.
    \end{itemize}
    Setting $\sigma_0 := (1- \sum_{k=1}^m\sigma_k) \chi_\Omega$ then yields $\sigma_0 \in\mathcal{D}(\R)$ with values in $[0,1]$ and $\supp \sigma_0\subset \Omega$ such that $\sum_{k=0}^m\sigma_k\equiv 1$ on $\Omega_\varepsilon$ with some $\varepsilon>0$.
    For $f\in \mathcal{D}'(\Omega)$ we therefore have
    $$
        f = f \sum_{k=0}^m \sigma_k \vert_\Omega
        = \sum_{k=0}^m \sigma_k \vert_\Omega f
        = \sum_{k=0}^m \sigma_k \vert_\Omega F_k\vert_\Omega
        = \sum_{k=0}^m (\sigma_k F_k)\vert_\Omega
        \qquad\text{in } \mathcal{D}'(\Omega)
    $$
    and hence 
    \begin{align*}
        \norm{f \sep \mathcal{E}^{s}_{u,p,q}(\Omega)}
        &= \norm{\sum_{k=0}^m (\sigma_k F_k)\vert_\Omega \sep \mathcal{E}^{s}_{u,p,q}(\Omega)} \\
        &\lesssim \norm{ (\sigma_0 F_0)\vert_\Omega \sep \mathcal{E}^{s}_{u,p,q}(\Omega)} + \sum_{k=1}^m \norm{ (\sigma_k F_k)\vert_\Omega \sep \mathcal{E}^{s}_{u,p,q}(\Omega)},
    \end{align*}
    where clearly $\norm{ (\sigma_0 F_0)\vert_\Omega \sep \mathcal{E}^{s}_{u,p,q}(\Omega)} \leq \norm{ \sigma_0 F_0 \sep \mathcal{E}^{s}_{u,p,q}(\R)}$.
    In order to bound the remaining terms, too, let $k\in \{1,\ldots,m\}$ be fixed. 
    Setting $\Omega_k:=\Phi_k(\Omega)$, \autoref{lem:tools_E}(ii) shows
    \begin{align*}
        \norm{ (\sigma_k F_k)\vert_\Omega \sep \mathcal{E}^{s}_{u,p,q}(\Omega)}
        &= \norm{ (\sigma_k F_k)\vert_{\Phi_k^{-1}(\Omega_k)} \sep \mathcal{E}^{s}_{u,p,q}(\Phi_k^{-1}(\Omega_k))} \\
        &\sim \norm{ (\sigma_k F_k)\vert_{\Phi_k^{-1}(\Omega_k)} \circ \Phi_k^{-1} \sep \mathcal{E}^{s}_{u,p,q}(\Omega_k)} \\
        &= \norm{ \big[(\sigma_k F_k) \circ \Phi_k^{-1}\big]\vert_{\Omega_k} \sep \mathcal{E}^{s}_{u,p,q}(\Omega_k)}.
    \end{align*}
    Since $S:=\supp \big[(\sigma_k F_k) \circ \Phi_k^{-1}\big]$ satisfies $S_\varepsilon\cap \Omega_k = S_\varepsilon\cap \omega_k$, we can apply \autoref{lem:tools_E}(iii) to conclude
    $$
        \norm{ (\sigma_k F_k)\vert_\Omega \sep \mathcal{E}^{s}_{u,p,q}(\Omega)}
        \sim \norm{ \big[(\sigma_k F_k) \circ \Phi_k^{-1}\big]\vert_{\omega_k} \sep \mathcal{E}^{s}_{u,p,q}(\omega_k)}
    $$
    which completes the proof of \eqref{eq:proof_splitting}.

    \emph{Step 2 (Switch to oscillation-based norms). }
    In order to apply \autoref{thm_osc_Rda=2} and \autoref{prop:osc_special} to \eqref{eq:proof_splitting}, we need to make sure that we actually deal with regular distributions. 
    To this end, again let $E$ be the trivial extension from $\Omega$ to $\R$; see \autoref{lem:tools_M}(i). 
    If we assume that $f \in L^\loc_{\max\{p,v\}}(\Omega)$ with $v\geq 1$, then $f$ is regular and  
    \begin{align*}
        f_0:=\sigma_0 Ef \in L^\loc_{\max\{p,v\}}(\R)
    \end{align*}
    satisfies $f_0 = \sigma_0 F_0$ in $\mathcal{S}'(\R)$, although $Ef$ might not be a valid choice for $F_0\in\mathcal{S}'(\R)$ (which explains the complicated detour). 
    Similarly, it is straightforward to check that
    \begin{align}\label{eq:proof_fk}
        f_k:=\big[(\sigma_k Ef) \circ \Phi_k^{-1}\big]\vert_{\omega_k}\in L^\loc_{\max\{p,v\}}(\omega_k), \qquad k=1,\ldots,m,
    \end{align}
    equals $\big[(\sigma_k F_k) \circ \Phi_k^{-1}\big]\vert_{\omega_k}$ in $\mathcal{D}'(\omega_k)$ for every choice of the extension $F_k$.
    Hence,
    $$
        \norm{f \sep \mathcal{E}^{s}_{u,p,q}(\Omega)}
        \lesssim \norm{ f_0 \sep \mathcal{E}^{s}_{u,p,q}(\R)} + \sum_{k=1}^m \norm{ f_k \sep \mathcal{E}^{s}_{u,p,q}(\omega_k)}
    $$
    with suitably localized $f_k\in L^\loc_{\max\{p,v\}}$  and special Lipschitz domains $\omega_k$. 
    
    Now let $\widetilde{N}:=N+L-1$ with $N\in\N$ and $L>s+d$ such that in particular $s<\widetilde{N}$. 
    If we additionally assume that $s>\sigma_{p,q}$, then \autoref{thm_osc_Rda=2} as well as \autoref{prop:osc_special} (with $R:=T:=v:=1$ and $\widetilde{N}$) imply
    \begin{align*}
        \norm{ f_0 \sep \mathcal{E}^{s}_{u,p,q}(\R)} 
        &\lesssim \norm{\int_{B(\,\cdot\,,1)} \abs{f_0(y)} \d y \sep \mathcal{M}^{u}_{p}(\R)} + \abs{f_0}^{(1,1,\widetilde{N})}_{\osc},\\
        \norm{f_k \sep \mathcal{E}^{s}_{u,p,q}(\omega_k)} 
        &\lesssim \norm{\int_{B(\,\cdot\,,1) \cap \omega_k} \abs{f_k(y)} \d y \sep \mathcal{M}^{u}_{p}(\omega_k)} + \abs{f_k}^{(1,1,\widetilde{N})}_{\osc,\omega_k}, \qquad k=1,\ldots,m,
    \end{align*}
    and it remains to estimate these terms by corresponding expressions of~$f$ on~$\Omega$.
    
    \emph{Step 3 (Estimates for $k>0$). }
    Let $k\in\{1,\ldots,m\}$ be fixed. 

    \emph{Substep 3a (Preparation). }
    Let us first show that there exists $r_k=r_k(\Omega)> 0$ such that 
    \begin{align}\label{eq:proof_avg_bound}
        \norm{ \int_{B(\,\cdot\,,r_k) \cap \omega_k} \abs{f_k(y)} \d y \sep \mathcal{M}^{u}_{p}(\omega_k)}
        &\lesssim \norm{ \int_{B(\,\cdot\,,R) \cap \Omega} \abs{f(y)} \d y  \sep \mathcal{M}^{u}_{p}(\Omega)}
    \end{align}
    with some constant that does not depend on $f$.
    
    To this end, note that by construction $\supp (f_k)\subseteq \Phi_k(B_k\cap \Omega)\subseteq \Phi_k(B_k) \cap \Phi_k(\Omega)$ has distance $r_k>0$ to $(\omega_k\setminus\Omega_k)\cup (\Omega_k\setminus\omega_k)$, where we recall that $\Omega_k=\Phi_k(\Omega)$.
    In particular, the integral
    $$
        A_k(x):=\int_{B(x,r_k) \cap \omega_k} \abs{f_k(y)} \d y = \int_{B(x,r_k) \cap \Omega_k} \abs{f_k(y)} \d y, \qquad x\in\R,
    $$
    vanishes for all $x\in\omega_k\setminus\Omega_k$.
    Hence, we can apply \autoref{lem:tools_M}(v) to derive
    \begin{align*}
        \norm{ A_k \sep \mathcal{M}^{u}_{p}(\omega_k)}
        &=\norm{ \chi_{\omega_k\cap\Omega_k}(\cdot) \left(\int_{B(\,\cdot\,,r_k) \cap \Omega_k} \abs{f_k(y)} \d y \right)\Big\vert_{\omega_k} \sep \mathcal{M}^{u}_{p}(\omega_k)} \\
        &\sim\norm{ \chi_{\omega_k\cap\Omega_k}(\cdot) \left(\int_{B(\,\cdot\,,r_k) \cap \Omega_k} \abs{f_k(y)} \d y \right)\Big\vert_{\Omega_k} \sep \mathcal{M}^{u}_{p}(\Omega_k)} \\
        &= \norm{ A_k \sep \mathcal{M}^{u}_{p}(\Omega_k)}.
    \end{align*}
    Next, \eqref{eq:proof_fk} and a transformation of measure yield
    \begin{align*}
        A_k(x) 
        &= \int_{B(x,r_k) \cap \Omega_k} \abs{(\sigma_k Ef) (\Phi_k^{-1}(y))} \d y\\
        &\sim \int_{\Phi_k^{-1}(B(x,r_k) \cap \Omega_k)} \abs{(\sigma_k Ef)(z)} \d z \\
        &\leq \int_{B(\Phi_k^{-1}(x),c\,r_k)\cap \Omega} \abs{f(z)} \d z, \qquad x\in\Omega_k,
    \end{align*}
    since $\Phi_k^{-1}(B(x,r_k) \cap \Omega_k) \subseteq B(\Phi_k^{-1}(x),c\,r_k)\cap \Omega$ for some $c=c_\Phi>0$, as well as $\sigma_k(z)\leq 1$ and $Ef(z)=f(z)$ for all $z\in\Omega$.
    Together this shows \eqref{eq:proof_avg_bound}:
    \begin{align*}
        \norm{ \int_{B(\,\cdot\,,r_k) \cap \omega_k} \abs{f_k(y)} \d y \sep \mathcal{M}^{u}_{p}(\omega_k)}
        &\lesssim \norm{ \left( \int_{B(\,\cdot\,,c\,r_k) \cap \Omega} \abs{f(z)} \d z \right)\circ \Phi_k^{-1} \sep \mathcal{M}^{u}_{p}(\Omega_k)} \\
        &\lesssim \norm{ \int_{B(\,\cdot\,,c\,r_k) \cap \Omega} \abs{f(z)} \d z  \sep \mathcal{M}^{u}_{p}(\Phi_k^{-1}(\Omega_k))} \\
        &\sim \norm{ \int_{B(\,\cdot\,,R) \cap \Omega} \abs{f(y)} \d y  \sep \mathcal{M}^{u}_{p}(\Omega)},
    \end{align*}
    where we used \autoref{lem:tools_M}(iv) as well as $\Phi_k^{-1}(\Omega_k)=\Omega$ and \autoref{lem:avg}.

    \emph{Substep 3b (Main term on $\omega_k$). }
    \autoref{lem:avg} combined with \eqref{eq:proof_avg_bound} from the previous substep immediately shows that
    \begin{align*}
        \norm{ \int_{B(\,\cdot\,,1) \cap \omega_k} \abs{f_k(y)} \d y \sep \mathcal{M}^{u}_{p}(\omega_k)}
        &\sim \norm{ \int_{B(\,\cdot\,,r_k) \cap \omega_k} \abs{f_k(y)} \d y \sep \mathcal{M}^{u}_{p}(\omega_k)} \\
        &\lesssim \norm{ \int_{B(\,\cdot\,,R) \cap \Omega} \abs{f(y)} \d y  \sep \mathcal{M}^{u}_{p}(\Omega)}.
    \end{align*}
    
    \emph{Substep 3c (Oscillation term on $\omega_k$). }
    Fix $0<t_k < \min\!\left\{1, \frac{T}{c}, \frac{R}{c}, \frac{r_k}{2} \right\}$ with $r_k$ and~$c$ as above.
    Then for $t_k<t \leq 1$ and $x\in\omega_k$ we can use $P:=0\in\mathcal{P}_{\widetilde{N}-1}$ to bound
    \begin{align*}
        \osc_{1,\omega_k}^{\widetilde{N}-1} f_k(x,t)
        = \inf_{P \in \mathcal{P}_{\widetilde{N}-1}} t^{-d} \int_{B(x,t) \cap \omega_k} \abs{f_k(y) - P(y)} \d y
        &\leq t_k^{-d} \int_{B(x,1) \cap \omega_k} \abs{f_k(y)} \d y 
    \end{align*}
    such that
    \begin{align*}
        I_1^{(k)} 
        &:= \norm{\Big( \int_{t_k}^1 \big[ t^{-s} \, \osc_{1,\omega_k}^{\widetilde{N}-1} f_k(\cdot,t) \big]^{q} \frac{\d t}{t} \Big)^{\frac{1}{q}} \sep \mathcal{M}^{u}_{p}(\omega_k)} \\
        &\lesssim \norm{\Big( \int_{t_k}^1 \Big[ t^{-s} \, \int_{B(\,\cdot\,,1) \cap \omega_k} \abs{f_k(y)} \d y \Big]^{q} \frac{\d t}{t} \Big)^{\frac{1}{q}} \sep \mathcal{M}^{u}_{p}(\omega_k)} \\
        &\sim \norm{ \int_{B(\,\cdot\,,1) \cap \omega_k} \abs{f_k(y)} \d y \sep \mathcal{M}^{u}_{p}(\omega_k)}
    \end{align*}
    which can be bounded as in Substep 3b.

    For the remaining part of the integral in $\abs{f_k}^{(1,1,\widetilde{N})}_{\osc,\omega_k}$ we again use the support properties of $f_k$. Since $t_k < \frac{r_k}{2}$ they imply that for all $0<t \leq t_k$ and $x\in\omega_k$
    \begin{align*}
        \osc_{1,\omega_k}^{\widetilde{N}-1} f_k(x,t) = \begin{cases}
            0, & x\in\omega_k\setminus\Omega_k,\\
            \displaystyle\inf_{P \in \mathcal{P}_{\widetilde{N}-1}} t^{-d} \int_{B(x,t) \cap \Omega_k} \abs{ f_k(y) - P(y)} \d y , & x\in\omega_k\cap\Omega_k.
        \end{cases}
    \end{align*}
    For every fixed $P\in\mathcal{P}_{\widetilde{N}-1}$ and $x\in\omega_k\cap\Omega_k$, Formula \eqref{eq:proof_fk} and a transformation of measure as above further yield
    \begin{align*}
        \int_{B(x,t) \cap \Omega_k} \abs{ f_k(y) - P(y)} \d y
        &= \int_{B(x,t) \cap \Omega_k} \abs{(\sigma_k Ef)(\Phi_k^{-1}(y)) - (P\circ \Phi_k)(\Phi_k^{-1}(y))} \d y \\
        &\sim \int_{\Phi_k^{-1}(B(x,t) \cap \Omega_k)} \abs{(\sigma_k Ef)(z) - (P\circ \Phi_k)(z)} \d z \\
        &\leq \int_{B(\Phi_k^{-1}(x), c\,t) \cap \Omega} \abs{(\sigma_k Ef)(z) - (P\circ \Phi_k)(z)} \d z,
    \end{align*}
    where we note that  $(P\circ \Phi_k)\in\mathcal{P}_{\widetilde{N}-1}$.
    Together this shows
    \begin{align*}
        \osc_{1,\omega_k}^{\widetilde{N}-1} f_k(x,t)
        &\lesssim \chi_{\omega_k\cap\Omega_k}(x) \inf_{\widetilde{P} \in \mathcal{P}_{\widetilde{N}-1}} (c\,t)^{-d} \int_{B(\Phi_k^{-1}(x),c\,t)\cap \Omega} \abs{(\sigma_k Ef)(z) - \widetilde{P}(z)} \d z  \\
        &=\chi_{\omega_k\cap\Omega_k}(x) \, E\Big(\osc_{1,\Omega}^{\widetilde{N}-1} \big[(\sigma_k Ef)\vert_\Omega\big](\Phi_k^{-1}(\cdot),c\,t)\Big)(x), \qquad x\in \omega_k,\, 0<t \leq t_k.
    \end{align*}
    So using \autoref{lem:tools_M}(v) and (iv) we can estimate
    \begin{align*}
        I_0^{(k)} 
        &:=\norm{\Big( \int_{0}^{t_k} \Big[ t^{-s} \, \osc_{1,\omega_k}^{\widetilde{N}-1} f_k(\cdot,t) \Big]^{q} \frac{\d t}{t} \Big)^{\frac{1}{q}} \sep \mathcal{M}^{u}_{p}(\omega_k)} \\
        &\lesssim \norm{\bigg[ \chi_{\omega_k\cap\Omega_k}(\cdot) \Big( \int_{0}^{t_k} \Big[ t^{-s} \, E\Big(\osc_{1,\Omega}^{\widetilde{N}-1} \big[(\sigma_k Ef)\vert_\Omega\big](\Phi_k^{-1}(\cdot),c\,t)\Big) \Big]^{q} \frac{\d t}{t} \Big)^{\frac{1}{q}} \bigg]\bigg\vert_{\omega_k} \sep \mathcal{M}^{u}_{p}(\omega_k)} \\
        &\sim \norm{\bigg[ \chi_{\omega_k\cap\Omega_k}(\cdot) \Big( \int_{0}^{t_k} \Big[ t^{-s} \, E\Big(\osc_{1,\Omega}^{\widetilde{N}-1} \big[(\sigma_k Ef)\vert_\Omega\big](\Phi_k^{-1}(\cdot),c\,t)\Big) \Big]^{q} \frac{\d t}{t} \Big)^{\frac{1}{q}} \bigg] \bigg\vert_{\Omega_k} \sep \mathcal{M}^{u}_{p}(\Omega_k)}\\
        &\lesssim \norm{ \Big( \int_{0}^{c\,t_k} \Big[ \tau^{-s} \, \osc_{1,\Omega}^{\widetilde{N}-1} \big[ (\sigma_k Ef)\vert_\Omega \big](\ast,\tau) \Big]^{q} \frac{\d \tau}{\tau} \Big)^{\frac{1}{q}} \circ \Phi_k^{-1}  \sep \mathcal{M}^{u}_{p}(\Omega_k)}\\
        &\sim \norm{\Big( \int_{0}^{c\,t_k} \Big[ t^{-s} \, \osc_{1,\Omega}^{\widetilde{N}-1} \big[ (\sigma_k Ef)\vert_\Omega \big](\cdot,t) \Big]^{q} \frac{\d t}{t} \Big)^{\frac{1}{q}} \sep \mathcal{M}^{u}_{p}(\Omega)}.
    \end{align*}
    In order to further bound this quantity in terms of an oscillation of $f$, we need to get rid of the smooth cut-off function $\sigma_k$.
    For this purpose, we now let $x\in\Omega$ and $0<t\leq c\,t_k$ be fixed and use the following idea due to Triebel \cite[p.191]{Tr92}:
    Let $T_k\in\mathcal{P}_{L-1}$ be the Taylor polynomial of degree $L-1$ of $\sigma_k\in\mathcal{D}(\R)$ around $x$. 
    If $R_k$ denotes its remainder, we have
    $$
        \sigma_k(y) = T_k(y) + R_k(y)
        \qquad\text{with}\quad 
        \abs{T_k(y)}\lesssim 1 \quad\text{and}\quad \abs{R_k(y)}\lesssim t^{L}, \qquad y\in B(x,t)\cap\Omega,
    $$
    with implied constants independent of $x$ and $t$.
    Hence, $\widetilde{N}=N+L-1$ yields
    \begin{align*}
        &\osc_{1,\Omega}^{\widetilde{N}-1} \big[(\sigma_k Ef)\vert_\Omega\big](x,t) \\
        &\quad=\inf_{\widetilde{P} \in \mathcal{P}_{N-1}} t^{-d} \int_{B(x,t) \cap \Omega} \abs{(T_k+R_k)(y)\, Ef(y) - T_k(y)\,\widetilde{P}(y)} \d y  \\
        &\quad\lesssim\inf_{\widetilde{P} \in \mathcal{P}_{N-1}} \Big( t^{-d} \int_{B(x,t) \cap \Omega} \abs{Ef(y) - \widetilde{P}(y)} \d y \Big) +   t^{-d} t^{L} \int_{B(x,t) \cap \Omega} \abs{Ef(y)} \d y  \\
        &\quad\leq \osc_{1,\Omega}^{N-1} f(x,t) + t^{L-d} \int_{B(x,c\,t_k)\cap \Omega} \abs{f(y)} \d y , \qquad x\in\Omega,\; 0<t\leq c\, t_k,
    \end{align*}
    where due to $L>s+d$ and $t_k< \frac{R}{c}$ there holds
    $$
        \Big( \int_{0}^{c\,t_k} \Big[ t^{-s} t^{L-d} \int_{B(x,c\,t_k)\cap \Omega} \abs{f(y)} \d y \Big]^{q} \frac{\d t}{t} \Big)^{\frac{1}{q}}
        \lesssim \int_{B(x,R)\cap \Omega} \abs{f(y)} \d y.
    $$
    Therefore, we derive
    \begin{align*}
        I_0^{(k)} &\lesssim \norm{\Big( \int_{0}^{c\,t_k} \Big[ t^{-s} \, \osc_{1,\Omega}^{N-1} f(\cdot,t) \Big]^{q} \frac{\d t}{t} \Big)^{\frac{1}{q}} \sep \mathcal{M}^{u}_{p}(\Omega)}
        + \norm{\int_{B(\,\cdot\,,R)\cap \Omega} \abs{f(y)} \d y \sep \mathcal{M}^{u}_{p}(\Omega)}
    \end{align*}
    and since $t_k < \frac{T}{c}$ this whole substep yields
    \begin{align*}
        \abs{f_k}^{(1,1,\widetilde{N})}_{\osc,\omega_k} 
        &\lesssim I_0^{(k)} + I_1^{(k)} 
        \lesssim \norm{\int_{B(\,\cdot\,,R)\cap \Omega} \abs{f(y)} \d y \sep \mathcal{M}^{u}_{p}(\Omega)} + \abs{f}^{(T,1,N)}_{\osc,\Omega}, 
        \qquad k=1,\ldots,m.
    \end{align*}
    
    \emph{Step 4 (Estimates for $k=0$). }
    We can follow the (complete) previous Step~3 line by line (formally setting $k:=0$, $\omega_k:=\R$, and $\Phi_k:=\mathrm{id}$ such that $c=1$) to show that also
    \begin{align*}
        \norm{\int_{B(\,\cdot\,,1)} \abs{f_0(y)} \d y \sep \mathcal{M}^{u}_{p}(\R)}
        &\lesssim \norm{\int_{B(\,\cdot\,,R)\cap \Omega} \abs{f(y)} \d y \sep \mathcal{M}^{u}_{p}(\Omega)} 
    \end{align*}
    as well as
    \begin{align*}
        \abs{f_0}^{(1,1,\widetilde{N})}_{\osc}  
        \lesssim I_0^{(0)} + I_1^{(0)} 
        \lesssim \norm{\int_{B(\,\cdot\,,R)\cap \Omega} \abs{f(y)} \d y \sep \mathcal{M}^{u}_{p}(\Omega)} + \abs{f}^{(T,1,N)}_{\osc,\Omega}.
    \end{align*}
    
    \emph{Step 5 (Conclusion). }
    A combination of Steps 1--4 yields the desired bound with $v=1$,
    \begin{align*}
        \norm{f \sep \mathcal{E}^{s}_{u,p,q}(\Omega)} 
        \lesssim \norm{\int_{B(\,\cdot\,,R)\cap \Omega} \abs{f(y)} \d y \sep \mathcal{M}^{u}_{p}(\Omega)} + \abs{f}^{(T,1,N)}_{\osc,\Omega},
    \end{align*}
    and the proof is finished by the arguments from Step~5 in the proof of \autoref{prop:diff_special}.
\end{proof}

\section{Characterization of \texorpdfstring{$\mathcal{E}^{s}_{u,p,q}(\Omega)$}{Esupq(Omega)} on Bounded Convex Lipschitz Domains via Differences}\label{Sec_Diff1_re}

In \autoref{Sec_Diff1} we already have proven a characterization in terms of higher order differences for Triebel-Lizorkin-Morrey spaces defined on special Lipschitz domains; cf.\ \autoref{thm_main_2}(ii). 
In what follows we will deduce counterparts for those results given that~$\Omega$ is a bounded convex Lipschitz domain. 
For that purpose, we combine our findings concerning local oscillations (\autoref{prop:boundedLip_osc_RTv}) with some specially tailored Whitney-type estimate proven in \cite{DekLev}. 
To formulate it, we require the following notation. 
For $d\in\N$ let $\Omega \subset \mathbb{R}^d$ be a domain and let $x \in \Omega$ and $h \in \mathbb{R}^d$.  
For $f \in L_{v}^{\loc}(\Omega)$ with $0 < v \leq \infty$ and $N\in\N$ we then put
\begin{align*}
        \Delta^{N}_{h,\Omega}f(x) := \begin{cases}
        \Delta^{N}_{h}f(x),      & \qquad [  x, x + Nh  ] \subset \Omega,  \\
         0,                      &  \qquad  \text{otherwise},
        \end{cases}
    \end{align*}
where $[a,b]$ denotes the line segment with end points $a$ and $b$. 

\begin{lem}[{\cite[Theorem 1.4]{DekLev}}]\label{lem_Whitney_est1}
    For $d\in\N$ let $\Omega \subset \mathbb{R}^d$ be a bounded convex Lipschitz domain, let $N \in \mathbb{N}$ and $0 < v \leq \infty$. 
    Then there exists a constant $C > 0$ independent of $\Omega$ such that for all $f \in L_{v}^{\loc}(\Omega)$
    \begin{align*}
        \inf_{P \in \mathcal{P}_{N-1}} \Big( \int_{\Omega} \abs{f(x) - P(x)}^{v} \d x \Big)^{\frac{1}{v}} 
        \leq C \sup_{\abs{h} \leq \mathrm{diam}(\Omega) } \Big(  \int_{\Omega} \abs{\Delta^{N}_{h,\Omega}f(x)}^{v} \d x  \Big )^{\frac{1}{v}}.
    \end{align*}
\end{lem}
There exist various different versions of \autoref{lem_Whitney_est1}. 
For example a corresponding statement for the special case that $\Omega$ is a cube is shown in \cite[Theorem A.1]{HN}. 
Here also a comprehensive discussion of the history of such Whitney-type estimates can be found. 

Now we are well-prepared to prove the following estimate for bounded convex Lipschitz domains. 
In combination with \autoref{prop:Omega_lower} it particularly proves \autoref{thm_main_2}(iii).
\begin{prop}
    For $d\in\N$ let $ \Omega \subset \mathbb{R}^{d}$ be a bounded convex Lipschitz domain. Further let $1 < p \leq u < \infty$, $1 < q \leq \infty$, $0 < T \leq \infty$, $0<R<\infty$,  $N\in \mathbb{N}$, and $s>0$. 
    Then for $f \in L^\loc_{\infty}(\Omega)$ there holds
    \begin{align*}
        \norm{ f \sep \mathcal{E}^{s}_{u,p,q}(\Omega) } 
        & \lesssim  \norm{ \esssup_{y\in B(\,\cdot\,, R) \cap \Omega} \abs{f(y)} \sep \mathcal{M}^{u}_{p}( \Omega)} + \abs{f}^{(T,\infty,N)}_{\Delta,\Omega} 
    \end{align*}
    as well as $\norm{ f \sep \mathcal{E}^{s}_{u,p,q}(\Omega) } \lesssim  \norm{ f \sep  \mathcal{M}^{u}_{p}( \Omega)} + \abs{f}^{(T,\infty,N)}_{\Delta,\Omega}$.
    In both cases, the implied constants are independent of $f$.
\end{prop}
\begin{proof}
    W.l.o.g.\ assume $q<\infty$ and choose $J\in\N$ with $2^{-J+2}\leq T$.
    Further note that for $v:=1$ our assumptions imply that $L^\loc_{\infty}(\Omega) \subset L^\loc_{\max\{p,v\}}(\Omega)$ and $\sigma_{p,q}=0< s$.
    Consequently, we can apply \autoref{prop:boundedLip_osc_RTv} (with $T:=2^{-J}$) to conclude
    $$
        \norm{ f \sep \mathcal{E}^{s}_{u,p,q}(\Omega)}
        \lesssim \norm{ f \sep  \mathcal{M}^{u}_{p}( \Omega)} + \abs{f}^{(2^{-J},1,N)}_{\osc,\Omega},
    $$
    where clearly $\norm{ f \sep  \mathcal{M}^{u}_{p}( \Omega)} \leq \big\Vert \esssup_{y\in B(\,\cdot\,, R) \cap \Omega} \abs{f(y)} \big| \mathcal{M}^{u}_{p}( \Omega) \big\Vert$.
    Therefore, it suffices to upper bound $\abs{f}^{(2^{-J},1,N)}_{\osc,\Omega}$ by $\abs{f}^{(T,\infty,N)}_{\Delta,\Omega}$.
    For this purpose, we employ the Whitney-type estimate \autoref{lem_Whitney_est1} for the bounded convex Lipschitz domains $B(x,t) \cap \Omega$ with
    $x \in \Omega$ and $0 < t \leq 2^{-J}$ which gives
    \begin{align*}
        t^{-s} \, \osc_{1,\Omega}^{N-1} f(x,t) 
        &= t^{-s-d} \inf_{P \in \mathcal{P}_{N-1}} \int_{B(x,t) \cap \Omega} \abs{f(y) - P(y)} \d y \\ 
        &\lesssim t^{-s-d}  \sup_{\abs{h} \leq \mathrm{diam}(B(x,t) \cap \Omega)}  \int_{B(x,t) \cap \Omega} \abs{\Delta^{N}_{h, B(x,t)\cap\Omega} f(y)} \d y \\
        &\lesssim (2t)^{-s} \abs{B(x,t)}^{-1} \int_{B(x,t) } \chi_{\Omega}(y) \esssup_{h\in V^N(y,2t)} \abs{\Delta^{N}_{h} f(y)} \d y \\
        &\leq \textit{\textbf{M}} \bigg( (2t)^{-s} \, \chi_{\Omega}(\ast) \esssup_{h\in V^N(\ast,2t)} \abs{\Delta^{N}_{h} f(\ast)} \bigg)(x)
    \end{align*}
    with \textit{\textbf{M}} being the Hardy-Littlewood maximal operator.
    We then find
    \begin{align*}
        \int_0^{2^{-J}} \big[ t^{-s} \, \osc_{1,\Omega}^{N-1} f(x,t) \big]^{q} \frac{\d t}{t}
        &= \sum_{j=J-1}^\infty \int_{2^{-(j+2)}}^{2^{-(j+1)}} \big[ t^{-s} \, \osc_{1,\Omega}^{N-1} f(x,t) \big]^{q} \frac{\d t}{t} \\
        &\lesssim \sum_{j=J-1}^\infty 2^{(j+1)sq} \, \big[ \osc_{1,\Omega}^{N-1} f(x,2^{-(j+1)}) \big]^{q}  \\
        &\lesssim \sum_{j=J-1}^\infty \bigg| \textit{\textbf{M}} \bigg( (2^{-j})^{-s} \, \chi_{\Omega}(\ast) \esssup_{h\in V^N(\ast,2^{-j})} \abs{\Delta^{N}_{h} f(\ast)} \bigg)(x) \bigg|^q 
    \end{align*}
    for every $x\in\Omega$ such that \autoref{lem:tools_M}(ii) and \autoref{l_ineq1} imply
    \begin{align*}
        \abs{f}^{(2^{-J},1,N)}_{\osc,\Omega}
        &\lesssim \norm{ \bigg( \sum_{j=J-1}^\infty \bigg| \textit{\textbf{M}} \bigg( (2^{-j})^{-s} \, \chi_{\Omega}(\ast) \esssup_{h\in V^N(\ast,2^{-j})} \abs{\Delta^{N}_{h} f(\ast)} \bigg)(\cdot) \bigg|^q \bigg)^{\frac{1}{q}} \sep \mathcal{M}^{u}_{p}(\R)} \\
        &\lesssim \norm{ \bigg( \sum_{j=J-1}^\infty \bigg| (2^{-j})^{-s} \, \chi_{\Omega}(\cdot) \esssup_{h\in V^N(\cdot,2^{-j})} \abs{\Delta^{N}_{h} f(\cdot)} \bigg|^q \bigg)^{\frac{1}{q}} \sep \mathcal{M}^{u}_{p}(\R)}.
    \end{align*}
    Finally, we employ \autoref{lem:tools_M}(v) to conclude
    \begin{align*}
        \abs{f}^{(2^{-J},1,N)}_{\osc,\Omega}
        &\leq \norm{ \bigg( \sum_{j=J-1}^\infty 2^{jsq} \Big[ \esssup_{h\in V^N(\cdot,2^{-j})} \abs{\Delta^{N}_{h} f(\cdot)} \Big]^q \bigg)^{\frac{1}{q}} \sep \mathcal{M}^{u}_{p}(\Omega)} \\
        &\lesssim \norm{ \bigg( \sum_{j=J-1}^\infty \!\int_{2^{-j}}^{2^{-(j-1)}}\!\! \Big[ t^{-s} \esssup_{h\in V^N(\cdot,t)}\! \abs{\Delta^{N}_{h} f(\cdot)}\! \Big]^q \frac{\d t}{t} \bigg)^{\frac{1}{q}} \sep \mathcal{M}^{u}_{p}(\Omega)} 
        = \abs{f}^{(2^{-J+2},\infty,N)}_{\Delta,\Omega}
    \end{align*}
    so that the use of $2^{-J+2}\leq T$ finishes the proof.
\end{proof}

\section{Summary and Further Issues}
Throughout this paper we obtained several characterizations in terms of local oscillations for the Triebel-Lizorkin-Morrey spaces $\mathcal{E}^{s}_{u,p,q}$ defined either on $\mathbb{R}^d$ or on special or bounded Lipschitz domains. 
Moreover, as a byproduct we also found new characterizations via differences of higher order. 
Nevertheless, there are still some open questions. Some of them (which also will be subject of future research) can be found in the following list:

\begin{enumerate}
    \item In our main results (Theorems \ref{mainresult1} and \ref{thm_main_2}) the additional condition $v \geq 1$ shows up if~$\Omega$ is a Lipschitz domain. 
    This restriction seems to have technical reasons only and stems from \autoref{prop:diff_special}. 
    Consequently, the natural question arises whether it is possible to modify the proof of \autoref{prop:diff_special} such that the condition $v \geq 1$ can be dropped.
    In this case, we could also drop the restriction $p\geq 1$ in part (ii) of Theorems \ref{mainresult1} and \ref{thm_main_2}, respectively; see Step~2 in the proof of \autoref{thm_osc_Rda=2}.
    
    \item If $\Omega$ is a bounded Lipschitz domain, we only have a characterization in terms of differences for $\mathcal{E}^{s}_{u,p,q}(\Omega)$ under very restrictive conditions on the parameters. 
    Moreover, $\Omega$ has to be convex. 
    Those restrictions are coming from the Whitney-type estimate given in \autoref{lem_Whitney_est1}. Therefore it would be desirable to have advanced counterparts of \autoref{lem_Whitney_est1} that hold for arbitrary bounded Lipschitz domains and where the supremum on the right-hand side is replaced by an integral.
    
    \item In our main results several conditions concerning the parameter $s$ show up, see \eqref{cond_on_s} and \eqref{cond_on_s2}. Some of them seem to be necessary. However, in particular if $p < 1$ or $q < 1$ there are still some open questions concerning necessity. 
    Consequently, we want to know whether these conditions are sharp and necessary. Some first results concerning this topic can be found in \cite{Ho1}, see also \cite{HoSi20} and \cite{HoN}.
\end{enumerate}

\medskip

\noindent
\textbf{Acknowledgments:} 
Marc Hovemann has been supported by Deutsche Forschungs\-gemeinschaft (DFG), grant DA 360/24-1.
Moreover, the authors are grateful to Winfried Sickel and Stephan Dahlke for several valuable discussions.


\addcontentsline{toc}{section}{References}
\small

\begin{thebibliography}{999}

\bibitem{BalDieWei}
A.Kh. Balci, L. Diening and M. Weimar, \emph{Higher order {C}alder\'{o}n-{Z}ygmund estimates for the {$p$}-{L}aplace equation}, J. Differential Equations \textbf{268} (2020), 590--635.

\bibitem{CaNoSte}
C. Canuto, R.H. Nochetto, R. Stevenson and M. Verani, \emph{Convergence and optimality of hp-AFEM}, Numer. Math. \textbf{135} (2017), 1073--1119.

\bibitem{CioWei}
P.A. Cioica-Licht and M. Weimar, \emph{On the limit regularity in {S}obolev and {B}esov scales related to approximation theory}, J. Fourier Anal. Appl. \textbf{26}(1), Art. 10 (2020), 24 pp.

\bibitem{DahDieHar+}
S. Dahlke, L. Diening, C. Hartmann, B. Scharf and M. Weimar, \emph{Besov regularity of solutions to the {$p$}-{P}oisson equation}, Nonlinear Anal. \textbf{130} (2016), 298--329.

\bibitem{DekLev}
S. Dekel and D. Leviatan, \emph{Whitney estimates for convex domains with applications to multivariate piecewise polynomial approximation}, Found. Comput. Math. \textbf{4} (2004), 345--368.

\bibitem{Dor1}
J.R. Dorronsoro, \emph{Poisson integrals of regular functions}, Trans. Amer. Math. Soc. \textbf{297} (1986), 669--685.

\bibitem{Dri1}
D. Drihem, \emph{Characterizations of Besov-type and Triebel-Lizorkin-type spaces by differences}, J. Funct. Spaces Appl. 2012, Article ID 328908, 24 pp.

\bibitem{DeSh}
R.A. DeVore and R.C. Sharpley, \emph{Maximal functions measuring smoothness}, Mem. Amer. Math. Soc. \textbf{47} (1984), no. 293, 115 pp.

\bibitem{GoHaSkr}
H.F. Goncalves, D.D. Haroske and L. Skrzypczak, \emph{Limiting embeddings of Besov-type and Triebel-Lizorkin-type spaces on domains and an extension operator}, to appear in: Ann. Mat. Pura Appl. (2023+), 36 pp.

\bibitem{HaMoSk}
D.D. Haroske, S.D. Moura and L. Skrzypczak, \emph{Smoothness Morrey spaces of regular distributions, and some unboundedness property}, Nonlinear Anal. \textbf{139} (2016), 218--244.

\bibitem{HaMoSk2}
D.D. Haroske, S.D. Moura and L. Skrzypczak, \emph{Some embeddings of Morrey spaces with critical smoothness}, J. Fourier Anal. Appl. \textbf{26}, Art. 50 (2020), 31 pp.

\bibitem{HarSchSkr18}
D.D. Haroske, C. Schneider and L. Skrzypczak, \emph{Morrey spaces on domains: different approaches and growth envelopes}. J. Geom. Anal. \textbf{28} (2018), 817--841.

\bibitem{HarSkrNuc}
D.D. Haroske and L. Skrzypczak, \emph{Nuclear embeddings of Morrey sequence spaces and smoothness Morrey spaces}, 
Preprint: arXiv:2211.02594v1 (2022), 23 pp.

\bibitem{HN}
L.I. Hedberg and Y.V. Netrusov, \emph{An axiomatic approach to function spaces, spectral synthesis, and Luzin approximation}, Mem. Amer. Math. Soc. \textbf{188} (2007), no. 882, 97 pp.

\bibitem{HoN}
M. Hovemann, \emph{Besov-Morrey spaces and differences}, Math. Rep. (Bucur.) \textbf{23}(73) (2021), 175--192.

\bibitem{H21}
M. Hovemann, \emph{Smoothness Morrey Spaces and Differences: Characterizations and Applications}. PhD thesis, FSU Jena, 2021.

\bibitem{Ho1}
M. Hovemann, \emph{Triebel-Lizorkin-Morrey spaces and differences}. Math. Nachr. \textbf{295}(4) (2022), 725--761.

\bibitem{HoSi20}
M. Hovemann and W. Sickel, \emph{Besov-type spaces and differences}, Eurasian Math. J. \textbf{13}(1) (2020), 25--56.

\bibitem{LiYYSaU}
Y. Liang, D. Yang, W. Yuan, Y. Sawano and T. Ullrich, \emph{A new framework for generalized Besov-type and Triebel-Lizorkin-type spaces}, Dissertationes Math. \textbf{489} (2013), 114 pp.

\bibitem{Liz1}
P.I. Lizorkin, \emph{Operators connected with fractional derivatives and classes of differentiable functions}, Trudy Mat. Inst. Steklov \textbf{117} (1972), 212--243.

\bibitem{Liz2}
P. I. Lizorkin, \emph{Properties of functions of the spaces $ \Lambda^{r}_{p,\theta}$}, Trudy Mat. Inst. Steklov  \textbf{131} (1974), 158--181.

\bibitem{maz}
A.L. Mazzucato, \emph{Decomposition of Besov-Morrey spaces}, in: W. Beckner, A. Nagel, A. Seeger and H.F. Smith (eds.), \emph{Harmonic Analysis at Mount Holyoke (South Hadley, MA, 2001)}, Contemp. Math. 320, Amer. Math. Soc., Providence, RI, 2003, 279--294.

\bibitem{Mor}
C.B. Morrey, \emph{On the solutions of quasi-linear elliptic partial differential equations},  Trans. Amer. Math. Soc. \textbf{43} (1937), 126--166.

\bibitem{Ry99}
V.S. Rychkov, \emph{On restrictions and extensions of the Besov and Triebel-Lizorkin spaces with respect to Lipschitz domains}, J. London Math. Soc. (2) \textbf{60} (1999), 237--257.

\bibitem{Saw10}
Y. Sawano, \emph{Besov-Morrey spaces and Triebel-Lizorkin-Morrey spaces on domains}, Math. Nachr. \textbf{283} (2010), 1456--1487.

\bibitem{See1}
A. Seeger, \emph{A note on Triebel-Lizorkin spaces}, in: Z. Ciesielski, (ed.), \emph{Approximation and Function Spaces}, Banach Center Publ. 22, PWN---Polish Scientific Publishers, Warsaw, 1989, 391--400.

\bibitem{Shv06}
P. Shvartsman, \emph{Local approximations and intrinsic characterization of spaces of smooth functions on regular subsets of {${\mathbb{R}}^n$}}, Math. Nachr. \textbf{279} (2006), 1212--1241.

\bibitem{Si1}
W. Sickel, \emph{Smoothness spaces related to Morrey spaces - A survey. I}, Eurasian Math. J. \textbf{3} (2012), 110--149.

\bibitem{Stein}
E.M. Stein, \emph{Singular Integrals and Differentiability Properties of Functions}. Princeton University Press, Princeton, 1970.

\bibitem{TangXu}
L. Tang and J. Xu, \emph{Some properties of Morrey type Besov-Triebel spaces}, Math. Nachr. \textbf{278} (2005), 904--917. 

\bibitem{Tr73}
H. Triebel, \emph{Spaces of distributions of Besov type on Euclidean $n$-space. Duality, interpolation}, Ark. Mat. \textbf{11} (1973), 13--64.

\bibitem{Tr83}
H. Triebel, \emph{Theory of Function Spaces}, Birkh{\"a}user, Basel, 1983.

\bibitem{Tr89}
H. Triebel, \emph{Local approximation spaces}, Z. Anal. Anwendungen \textbf{8}(3) (1989), 261--288.

\bibitem{Tr92}
H. Triebel, \emph{Theory of Function Spaces II}, Birkh{\"a}user, Basel, 1992.

\bibitem{Tr06} 
H. Triebel, \emph{Theory of Function Spaces III}, Birkh\"auser, Basel, 2006.

\bibitem{Tr14} 
H. Triebel, \emph{Hybrid Function Spaces, Heat and Navier-Stokes Equations}, EMS Tracts in Mathematics, vol. 24, European Mathematical Society, Z\"urich, 2014.

\bibitem{Yab}
K. Yabuta, \emph{Singular integral operators on Triebel-Lizorkin spaces}, Bull. Fac. Sci. Ibaraki Univ. Ser. A \textbf{20} (1988), 9--17.

\bibitem{yy1}
D.~Yang and W.~Yuan, \emph{A new class of function spaces connecting Triebel-Lizorkin spaces and $Q$ spaces}, J. Funct. Anal. \textbf{255} (2008), 2760--2809.

\bibitem{yy2}
D.~Yang and W.~Yuan, \emph{New Besov-type spaces and Triebel-Lizorkin-type spaces including $Q$ spaces}, Math. Z. \textbf{265}(2) (2010), 451--480.

\bibitem{Yao}
L. Yao, \emph{Some intrinsic characterizations of Besov-Triebel-Lizorkin-Morrey-type spaces on Lipschitz domains}. J. Fourier Anal. Appl. \textbf{29}, Art. 24 (2023), 21 pp.

\bibitem{ysy} 
W.~Yuan, W.~Sickel and D.~Yang, \emph{Morrey and Campanato Meet Besov, Lizorkin and Triebel}, Lecture Notes in Mathematics Vol. 2005, Springer, Berlin, 2010.

\bibitem{YSY2}
W. Yuan, W. Sickel and D. Yang, \emph{Interpolation of Morrey-Campanato and related smoothness spaces}, Sci. China Math \textbf{58}(9) (2015), 1835--1908.

\bibitem{ZHS}
C. Zhuo, M. Hovemann and W. Sickel, \emph{Complex interpolation of Lizorkin-Triebel-Morrey spaces on domains}, Anal. Geom. Metr. Spaces \textbf{8} (2020), 268-304.

\end{thebibliography}
\end{document}